\documentclass{amsproc}
\usepackage{amsmath}
\usepackage{amsfonts}

\theoremstyle{plain}
\newtheorem{acknowledgement}{Acknowledgement}

\newtheorem{conjecture}{Conjecture}
\newtheorem{corollary}{Corollary}

\newtheorem{definition}{Definition}

\newtheorem{lemma}{Lemma}

\newtheorem{proposition}{Proposition}
\newtheorem{remark}{Remark}

\newtheorem{theorem}{Theorem}
\numberwithin{equation}{section}
\input{tcilatex}

\begin{document}
\title[Multidimensional $q$-Normal]{Multidimensional $q$-Normal and related
distributions - Markov case}
\author{Pawe\l\ J. Szab\l owski}
\address{Department of Mathematics and Information Sciences\\
Warsaw University of Technology\\
pl. Politechniki 1, 00-661 Warsaw, Poland\\
}
\email{pawel.szablowski@gmail.com}
\date{february 17, 2010}

\begin{abstract}
We define and study distributions in $\mathbb{R}^{d}$ that we call $q-$%
Normal. For $q=1$ they are really multidimensional Normal, for $q\in (-1,1)$
they have densities, compact support and many properties that resemble
properties of ordinary multidimensional Normal distribution. We also
consider some generalizations of these distributions and indicate close
relationship of these distributions to Askey-Wilson weight function i.e.
weight with respect to which Askey-Wilson polynomials are orthogonal and
prove some properties of this weight function. In particular we prove a
generalization of Poisson-Mehler expansion formula.
\end{abstract}

\maketitle

\section{Introduction}

The aim of this paper is to define, analyze and possibly 'accustom' new
distributions in $\mathbb{R}^{d}$. They are defined with a help of two
one-dimensional distributions that first appeared recently, partially in
noncommutative context and are defined through infinite products. That is
why it is difficult to analyze them straightforwardly using ordinary
calculus. One has to refer to some extent to notations and results of so
called $q-$series theory.

However the distributions we are going to define and examine have \emph{%
purely commutative, classical probabilistic meaning}. They appeared first in
an excellent paper of Bo\.{z}ejko et al. \cite{Bo} as a by product of
analysis of some non-commutative model. Later they also appeared in purely
classical context of so called one-dimensional random fields first analyzed
by W. Bryc at al. in \cite{bryc1} and \cite{bms}. From these papers we can
deduce much information on these distributions. In particular we are able to
indicate sets of polynomials that are orthogonal with respect to measures
defined by these distributions. Those are so called $q-$Hermite and
Al-Salam-Chihara polynomials - a generalizations of well known sets of
polynomials. Thus in particular we know all moments of the discussed
one-dimensional distributions.

What is interesting about distributions discussed in this paper is that many
of their properties resemble similar properties of normal distribution. As
stated in the title we consider three families of distributions, however
properties of one, called multidimensional $q-$Normal, are main subject of
the paper. The properties of the remaining two are in fact only sketched.

All distributions considered in this paper have densities. The distributions
in this paper are parametrized by several parameters. One of this
parameters, called $q,$ belongs to $(-1,1]$ and for $q\allowbreak
=\allowbreak 1$ the distributions considered in this paper become ordinary
normal. Two out of three families of distributions defined in this paper
have the property that all their marginals belong to the same class as the
joint, hence one of the important properties of normal distribution.
Conditional distributions considered in this paper have the property that
conditional expectation of a polynomial is also a polynomial of the same
order - one of the basic properties of normal distributions. Distributions
considered in this paper satisfy Gebelein inequality -property discovered
first in the normal distribution context. Furthermore as in the normal case
lack of correlation between components of a random vectors considered in the
paper lead to independence of these components. Finally conditional
distribution $f_{C}\left( x|y,z\right) $ considered in this paper can be
expanded in series of the form $f_{C}\left( x|y,z\right) \allowbreak
=\allowbreak f_{M}\left( x\right) \sum_{i=0}^{\infty }h_{i}(x)g_{i}\left(
y,z\right) $ where $f_{M}$ is a marginal density, $\left\{ h_{i}\right\} $
are orthogonal polynomials of $f_{M}$ and $g_{i}\left( y,z\right) $ are also
polynomials. In particular if $f_{C}\left( x|y,z\right) \allowbreak
=\allowbreak f_{C}\left( x|q\right) $ that is when instead of conditional
distribution of $X|Y,Z$ we consider only distribution of $X|Y$ then $%
g_{i}\left( y\right) \allowbreak =\allowbreak h_{i}\left( y\right) $. In
this case such expansion formula it is a so called Poisson-Mehler formula, a
generalization of a formula with $h_{i}$ being ordinary Hermite polynomials
and $f_{M}\left( x\right) \allowbreak =\allowbreak \exp (-x^{2}/2)/\sqrt{%
2\pi }$ that appeared first in the normal distribution context.

On the other hand one of the conditional distributions that can be obtained
with the help of distributions considered in this paper is in fact a
re-scaled and normalized (that is multiplied by a constant so its integral
is equal to $1$) Askey-Wilson weight function. Hence we are able to prove
some properties of this Askey-Wilson density. In particular we will obtain a
generalization of Poisson-Mehler expansion formula for this density.

To define briefly and swiftly these one-dimensional distributions that will
be later used to construct multidimensional generalizations of normal
distributions, let us define the following sets 
\begin{equation*}
S\left( q\right) =\left\{ 
\begin{array}{ccc}
\lbrack -2/\sqrt{1-q},2/\sqrt{1-q}] & if & \left\vert q\right\vert <1 \\ 
\left\{ -1,1\right\} & if & q=-1%
\end{array}%
\right. .
\end{equation*}%
Let us set also $m\allowbreak +\allowbreak S\left( q\right) \allowbreak 
\overset{df}{=}\allowbreak \{x\allowbreak =\allowbreak m+y,y\in S\left(
q\right) \}$ and $\mathbf{m+S}\left( q\right) \allowbreak \overset{df}{=}%
\allowbreak (m_{1}\allowbreak +\allowbreak S\left( q\right) )\times
\allowbreak \ldots \allowbreak \times (m_{d}+S\left( q\right) )$ if $\mathbf{%
m\allowbreak =\allowbreak (}m_{1},\ldots ,m_{d})$. Sometimes to simplify
notation we will use so called indicator functions 
\begin{equation*}
I_{A}\left( x\right) \allowbreak =\allowbreak \left\{ 
\begin{array}{ccc}
1 & if & x\in A \\ 
0 & if & x\notin A%
\end{array}%
\right. .
\end{equation*}
The two one-dimensional distributions (in fact families of distributions)
are given by their densities.

The first one has density:%
\begin{equation}
f_{N}\left( x|q\right) =\frac{\sqrt{1-q}}{2\pi \sqrt{4-(1-q)x^{2}}}%
\prod_{k=0}^{\infty }\left( (1+q^{k})^{2}-(1-q)x^{2}q^{k}\right)
\prod_{k=0}^{\infty }(1-q^{k+1})I_{S\left( q\right) }\left( x\right)
\label{qN}
\end{equation}%
defined for $\left\vert q\right\vert <1,$ $x\in \mathbb{R}$. We will set
also 
\begin{equation}
f_{N}\left( x|1\right) \allowbreak =\allowbreak \frac{1}{\sqrt{2\pi }}\exp
\left( -x^{2}/2\right) .  \label{q=1}
\end{equation}%
For $q\allowbreak =\allowbreak -1$ considered distribution does not have
density, is discrete with two equal mass points at $S\left( -1\right) $.
Since this case leads to non-continuous distributions we will not analyze it
in the sequel.

The fact that such definition is reasonable i.e. that distribution defined
by $f_{N}\left( x|q\right) $  tends to normal $N\left( 0,1\right) $ as $%
q\longrightarrow 1^{-}$ will be justified in the sequel. The distribution
defined by $f_{N}\left( x|q\right) ,$ $-1<q\leq 1$ will be referred to as $q-
$Normal distribution.

The second distribution has density: 
\begin{subequations}
\begin{gather}
f_{CN}\left( x|y,\rho ,q\right) =\frac{\sqrt{1-q}}{2\pi \sqrt{4-(1-q)x^{2}}}%
\times  \label{1} \\
\prod_{k=0}^{\infty }\frac{(1-\rho ^{2}q^{k})\left( 1-q^{k+1}\right) \left(
(1+q^{k})^{2}-(1-q)x^{2}q^{k}\right) }{(1-\rho ^{2}q^{2k})^{2}-(1-q)\rho
q^{k}(1+\rho ^{2}q^{2k})xy+(1-q)\rho ^{2}(x^{2}+y^{2})q^{2k}}I_{S\left(
q\right) }\left( x\right)  \label{2}
\end{gather}%
defined for $\left\vert q\right\vert <1,$ $\left\vert \rho \right\vert <1$, $%
x\in \mathbb{R},$ $y\in S\left( q\right) $. It will be referred to as $%
(y,\rho ,q)-$Conditional Normal, distribution. For $q\allowbreak
=\allowbreak 1$ we set 
\end{subequations}
\begin{equation*}
f_{CN}\left( x|y,\rho ,1\right) \allowbreak =\allowbreak \frac{1}{\sqrt{2\pi
\left( 1-\rho ^{2}\right) }}\exp \left( -\frac{\left( x-\rho y\right) ^{2}}{%
2\left( 1-\rho ^{2}\right) }\right)
\end{equation*}%
(in the sequel we will justify this fact). Notice that we have $f_{CN}\left(
x|y,0,q\right) \allowbreak =\allowbreak f_{N}\left( x|q\right) $ for all $%
y\in S\left( q\right) $.

The simplest example of multidimensional density that can be constructed
from these two distribution is two dimensional density 
\begin{equation*}
g\left( x,y|\rho ,q\right) =f_{CN}\left( x|y,\rho ,q\right) f_{N}\left(
y|q\right) ,
\end{equation*}%
that will be referred to in the sequel as $N_{2}\left( 0,0,1,1,\rho
|q\right) $. Below we give some examples of plots of these densities. One
can see from these pictures how large and versatile family of distributions
is this family\newline
. \FRAME{itbpFU}{3.0727in}{2.4122in}{0in}{\Qcb{$\protect\rho =.5,$ $q=.8$}}{%
\Qlb{fig1}}{Figure}{\special{language "Scientific Word";type
"GRAPHIC";maintain-aspect-ratio TRUE;display "USEDEF";valid_file "T";width
3.0727in;height 2.4122in;depth 0in;original-width 7.6113in;original-height
5.9656in;cropleft "0";croptop "1";cropright "1";cropbottom "0";tempfilename
'KXZ8XE00.wmf';tempfile-properties "XPR";}}\newline
\FRAME{dtbpFU}{2.8191in}{2.153in}{0pt}{\Qcb{$\protect\rho =-.6,$ $q=-.7$}}{%
\Qlb{fig2}}{Figure}{\special{language "Scientific Word";type
"GRAPHIC";maintain-aspect-ratio TRUE;display "USEDEF";valid_file "T";width
2.8191in;height 2.153in;depth 0pt;original-width 8.1149in;original-height
6.1879in;cropleft "0";croptop "1";cropright "1";cropbottom "0";tempfilename
'KXZ8XE01.wmf';tempfile-properties "XPR";}}

It has compact support equal to $S\left( q\right) \times S\left( q\right) $
and two parameters. One playing similar r\^{o}le to parameter $\rho $ in
two-dimensional Normal distribution. The other parameter $q$ has a different
r\^{o}le. In particular it is responsible for modality of the distribution
and of course it defines its support.

As stated above, distribution defined by $f_{N}\left( x|q\right) $ appeared
in 1997 in \cite{Bo} in basically non-commutative context. It turns out to
be important both for classical and noncommutative probabilists as well as
for physicists. This distribution has been 'accustomed' i.e. equivalent form
of the density and methods of simulation of i.i.d. sequences drawn from it
are e.g. presented in \cite{szab2}. Distribution $f_{CN},$ although known
earlier in nonprobabilistic context, appeared (as an important probability
distribution) in the paper of W. Bryc \cite{bryc1} in a classical context as
a conditional distribution of certain Markov sequence. In the following
section we will briefly recall basic properties of these distributions as
well as of so called $q-$Hermite polynomials (a generalization of ordinary
Hermite polynomials). To do this we have to refer to notation and some of
the results of $q-$series theory.

The paper is organized as follows. In section 2 after recall some of the
results of $q-$series theory we present definition of multivariate $q-$%
Normal distribution. The following section presents main result. The last
section contains lengthy proofs of the results from previous section.

\section{Definition of multivariate $q$-Normal and some related distributions%
}

\subsection{Auxiliary results}

We will use traditional notation of $q-$series theory i.e. $\left[ 0\right]
_{q}\allowbreak =\allowbreak 0;$ $\left[ n\right] _{q}\allowbreak
=\allowbreak 1+q+\ldots +q^{n-1}\allowbreak =\allowbreak \frac{1-q^{n}}{1-q}%
, $ $\left[ n\right] _{q}!\allowbreak =\allowbreak \prod_{i=1}^{n}\left[ i%
\right] _{q},$ with $\left[ 0\right] _{q}!\allowbreak =1,\QATOPD[ ] {n}{k}%
_{q}\allowbreak =\allowbreak \left\{ 
\begin{array}{ccc}
\frac{\left[ n\right] _{q}!}{\left[ n-k\right] _{q}!\left[ k\right] _{q}!} & 
, & n\geq k\geq 0 \\ 
0 & , & otherwise%
\end{array}%
\right. $. It will be useful to use so called $q-$Pochhammer symbol for $%
n\geq 1:\left( a|q\right) _{n}=\prod_{i=0}^{n-1}\left( 1-aq^{i}\right) ,$
with $\left( a|q\right) _{0}=1$ , $\left( a_{1},a_{2},\ldots ,a_{k}|q\right)
_{n}\allowbreak =\allowbreak \prod_{i=1}^{k}\left( a_{i}|q\right) _{n}$.
Often $\left( a|q\right) _{n}$ as well as $\left( a_{1},a_{2},\ldots
,a_{k}|q\right) _{n}$ will be abbreviated to $\left( a\right) _{n}$ and $%
\left( a_{1},a_{2},\ldots ,a_{k}\right) _{n},$ if it will not cause
misunderstanding.

It is easy to notice that $\left( q\right) _{n}=\left( 1-q\right) ^{n}\left[
n\right] _{q}!$ and that\newline
$\QATOPD[ ] {n}{k}_{q}\allowbreak =$\allowbreak $\allowbreak \left\{ 
\begin{array}{ccc}
\frac{\left( q\right) _{n}}{\left( q\right) _{n-k}\left( q\right) _{k}} & ,
& n\geq k\geq 0 \\ 
0 & , & otherwise%
\end{array}%
\right. $.

Let us also introduce two functionals defined on functions $g:\allowbreak 
\mathbb{R\allowbreak \longrightarrow \allowbreak \mathbb{C}}\allowbreak $, 
\begin{equation*}
\left\Vert g\right\Vert _{L}^{2}\allowbreak =\allowbreak \int_{\mathbb{R}%
}\left\vert g\left( x\right) \right\vert ^{2}f_{N}\left( x\right)
dx,~\left\Vert g\right\Vert _{CL}^{2}\allowbreak =\allowbreak \int_{\mathbb{R%
}}\left\vert g\left( x\right) \right\vert ^{2}f_{CN}\left( x|y,\rho
,q\right) dx
\end{equation*}%
and sets: 
\begin{eqnarray*}
L\left( q\right) \allowbreak &=&\allowbreak \left\{ g:\mathbb{R}%
\longrightarrow \mathbb{C}:\left\Vert g\right\Vert _{L}<\infty \right\} , \\
CL\left( y,\rho ,q\right) \allowbreak &=&\allowbreak \{g:\mathbb{R}%
\longrightarrow \mathbb{C}\allowbreak :\allowbreak \left\Vert g\right\Vert
_{CL}<\infty \}.
\end{eqnarray*}%
Spaces $(L\left( q\right) ,\left\Vert .\right\Vert _{L})$ and $\left(
CL\left( y,\rho ,q\right) ,\left\Vert .\right\Vert _{CL}\right) $ are
Hilbert spaces with the usual definition of scalar product.

Let us also define the following two sets of polynomials:

-the $q-$Hermite polynomials defined by 
\begin{equation}
H_{n+1}(x|q)=xH_{n}(x|q)-[n]_{q}H_{n-1}(x|q),  \label{He}
\end{equation}%
for $n\geq 1$ with $H_{-1}(x|q)=0,$ $H_{0}(x|q)=1,$ and

-the so called Al-Salam-Chihara polynomials defined by the relationship for $%
n\geq 0:$%
\begin{equation}
P_{n+1}(x|y,\rho ,q)=(x-\rho yq^{n})P_{n}(x|y,\rho ,q)-(1-\rho
^{2}q^{n-1})[n]_{q}P_{n-1}(x|y,\rho ,q),  \label{AlSC}
\end{equation}%
with $P_{-1}\left( x|y,\rho ,q\right) \allowbreak =\allowbreak 0,$ $%
P_{0}\left( x|y,\rho ,q\right) \allowbreak =\allowbreak 1$.

Polynomials (\ref{He}) satisfy the following very useful identity originally
formulated for so called continuous $q-$Hermite polynomials $h_{n}$ (can be
found in e.g. \cite{IA} Thm. 13.1.5) and here below presented for
polynomials $H_{n}$ using the relationship 
\begin{equation}
h_{n}\left( x|q\right) \allowbreak =\allowbreak \left( 1-q\right)
^{n/2}H_{n}\left( \frac{2x}{\sqrt{1-q}}|q\right) ,~~n\geq 1,  \label{q-cont}
\end{equation}%
\begin{equation}
H_{n}\left( x|q\right) H_{m}\left( x|q\right) =\sum_{j=0}^{\min \left(
n,m\right) }\QATOPD[ ] {m}{j}_{q}\QATOPD[ ] {n}{j}_{q}\left[ j\right]
_{q}!H_{n+m-2k}\left( x|q\right) .  \label{identity}
\end{equation}

It is known (see e.g. \cite{bryc1}) that $q-$Hermite polynomials constitute
an orthogonal base of $L\left( q\right) $ while from \cite{bms} one can
deduce that $\left\{ P_{n}\left( x|y,\rho ,q\right) \right\} _{n\geq -1}$
constitute an orthogonal base of $CL\left( y,\rho ,q\right) $. Thus in
particular $0=\allowbreak \int_{S\left( q\right) }P_{1}\left( x|y,\rho
,q\right) f_{CN}\left( x|y,\rho ,q\right) dx\allowbreak =\allowbreak \mathbb{%
E}\left( X|Y=y\right) \allowbreak -\allowbreak \rho y$. Consequently, if $Y$
has also $q-$Normal distribution, then $\mathbb{E}XY\allowbreak =\allowbreak
\rho $.

It is known (see e.g. \cite{IA} formula 13.1.10) that%
\begin{equation}
\left\vert H_{n}\left( x|q\right) \right\vert \leq W_{n}\left( q\right)
\left( 1-q\right) ^{-n/2},  \label{ogr_H}
\end{equation}%
where 
\begin{equation}
W_{n}\left( q\right) \allowbreak =\allowbreak \sum_{i=0}^{n}\QATOPD[ ] {n}{i}%
_{q}.  \label{Wn}
\end{equation}

We will also use Chebyshev polynomials of the second kind $U_{n}\left(
x\right) $, that is $U_{n}\left( \cos \theta \right) =\frac{\sin \left(
n+1\right) \theta }{\sin \theta }$ and ordinary (probabilistic) Hermite
polynomials $H_{n}\left( x\right) $ i.e. polynomials orthogonal with respect
to $\frac{1}{\sqrt{2\pi }}\exp (-x^{2}/2)$. They satisfy $3-$term
recurrences: 
\begin{eqnarray}
2xU_{n}\left( x\right) &=&U_{n+1}\left( x\right) +U_{n-1}\left( x\right) ,
\label{_0} \\
xH_{n}\left( x\right) &=&H_{n+1}\left( x\right) +nH_{n-1}  \label{_1}
\end{eqnarray}%
with $U_{-1}\left( x\right) \allowbreak =\allowbreak H_{-1}(x)\allowbreak
=\allowbreak 0,$ $U_{0}\left( x\right) \allowbreak =\allowbreak H_{1}\left(
x\right) \allowbreak =\allowbreak 1.$

Some immediate observations concerning $q$-Normal and $(y,\rho ,q)-$%
Conditional Normal distributions are collected in the following Proposition:

\begin{proposition}
\label{uwaga}$1.$ $f_{CN}\left( x|y,0,q\right) =f_{N}(x|q).$

$2.$ $\forall n\geq 0:H_{n}\left( x|0\right) =U_{n}\left( x/2\right) ,$ $%
H_{n}\left( x|1\right) \allowbreak =\allowbreak H_{n}\left( x\right) .$

$3.$ $\forall n\geq 0:P_{n}\left( x|y,0,q\right) =H_{n}(x|q),$ $%
P_{n}(x|y,\rho ,1)\allowbreak =\allowbreak (1-\rho ^{2})^{n/2}H_{n}\left( 
\frac{x-\rho y}{\sqrt{1-\rho ^{2}}}\right) ,$ $P_{n}\left( x|y,\rho
,0\right) \allowbreak =\allowbreak U_{n}\left( x/2\right) \allowbreak
-\allowbreak \rho yU_{n-1}\left( x/2\right) \allowbreak +\allowbreak \rho
^{2}U_{n-2}\left( x/2\right) .$

$4.$ $f_{N}\left( x|0\right) =\frac{1}{2\pi }\sqrt{4-x^{2}}I_{<-2,2>}\left(
x\right) ,$ $f_{N}\left( x|q\right) \allowbreak \underset{q\rightarrow 1^{-}}%
{\longrightarrow }\allowbreak \frac{1}{\sqrt{2\pi }}\exp \left(
-x^{2}/2\right) $ pointwise.

$5.$ $f_{CN}\left( x|y,\rho ,0\right) \allowbreak =\allowbreak \frac{\left(
1-\rho ^{2}\right) \sqrt{4-x^{2}}}{2\pi \left( \left( 1-\rho ^{2}\right)
^{2}-\rho \left( 1+\rho ^{2}\right) xy+\rho ^{2}\left( x^{2}+y^{2}\right)
\right) }I_{<-2,2>}\left( x\right) ,$ $f_{CN}\left( x|y,\rho ,q\right)
\allowbreak \underset{q\rightarrow 1^{-}}{\longrightarrow }\allowbreak \frac{%
1}{\sqrt{2\pi (1-\rho ^{2})}}\exp \left( -\frac{\left( x-\rho y\right) ^{2}}{%
2\left( 1-\rho ^{2}\right) }\right) $ pointwise.
\end{proposition}

\begin{proof}
$1.$ Is obvious. $2.$ Follows observation that (\ref{He}) simplifies to (\ref%
{_0}) and (\ref{_1}) for $q\allowbreak =\allowbreak 0$ and $q\allowbreak
=\allowbreak 1$ respectively. $3.$ First two assertions follow either direct
observation in case of $P_{n}\left( x|y,\rho ,0\right) $ or comparison of (%
\ref{AlSC}) and (\ref{_1}) considered for $x\longrightarrow (x-\rho y)/\sqrt{%
1-\rho ^{2}}$ and then multiplication of both sides by $\left( 1-\rho
^{2}\right) ^{(n+1)/2}.$\ Third assertion follows following observations: $%
P_{-1}\left( x|y,\rho ,0\right) =0,$ $P_{0}\left( x|y,\rho ,0\right) =1,$ $%
P_{1}\left( x|y,\rho ,0\right) =x-\rho y$ , $P_{2}\left( x|y,\rho ,0\right)
\allowbreak =\allowbreak x(x-\rho y)\allowbreak -\allowbreak \left( 1-\rho
^{2}\right) ,$ $P_{n+1}\left( x|y,\rho ,0\right) \allowbreak =\allowbreak
xP_{n}\left( x|y,\rho ,0\right) \allowbreak -\allowbreak P_{n-1}\left(
x;y,\rho ,0\right) $ for $n\geq 1$ which is an equation (\ref{_0}) with $x$
replaced by $x/2$.

$4.$ $5.$ First assertions are obvious. Rigorous prove of pointwise
convergence of respective densities can be found in work of \cite{ISV87}. To
support intuition we will sketch the proof of convergence in distribution of
respective distributions. To do this we apply $2.$ and $3$. and see that $%
\forall n\geq 1$ $H_{n}\left( x|q\right) \allowbreak \longrightarrow
\allowbreak H_{n}\left( x\right) ,$ and $P_{n}(x|y,\rho ,q)\allowbreak
\longrightarrow \allowbreak (1-\rho ^{2})^{n/2}H_{n}\left( \frac{x-\rho y}{%
\sqrt{1-\rho ^{2}}}\right) $ as $q\rightarrow 1^{-}$. Now keeping in mind
that families $\left\{ H_{n}\left( x|q\right) \right\} _{n\geq 0}$ and $%
\left\{ P_{n}\left( x|y,\rho ,q\right) \right\} _{\geq 0}$ are orthogonal
with respect to distributions defined by respectively $f_{N}$ and $f_{CN}$
we deduce that distributions defined by $f_{N}$ and $f_{CN}$ tend to normal $%
N\left( 0,1\right) $ and $N\left( \rho y,\left( 1-\rho ^{2}\right) \right) $
distributions weakly as $q\longrightarrow 1^{-}$ since both $N\left(
0,1\right) $ and $N\left( \rho y,\left( 1-\rho ^{2}\right) \right) $ are
defined by their moments, which are defined by polynomials $H_{n},$ and $%
P_{n}.$
\end{proof}

\subsection{Multidimensional $q-$Normal and related distributions}

Before we present definition of the multidimensional $q-$Normal and related
distributions, let us generalize the two discussed above one-dimensional
distributions by introducing $(m,\sigma ^{2},q)-$Normal distribution as the
distribution with the density $f_{N}(\left( x-m\right) /\sigma |q)/\sigma $
for $m\in \mathbb{R},$ $\sigma >0,$ $q\in (-1,1]$. That is if $X\allowbreak
\sim \allowbreak (m,\sigma ^{2},q)-$Normal then $(X-m)/\sigma \allowbreak
\sim \allowbreak q-$Normal.

Similarly let us extend definition of $(y,\rho ,q)-$Conditional Normal by
introducing for $m\in \mathbb{R},$ $\sigma >0,$ $q\in (-1,1],$ $\left\vert
\rho \right\vert <1,$ $(m,\sigma ^{2},y,\rho ,q)$-Conditional Normal
distribution as the distribution whose density is equal to $f_{CN}\left(
(x-m)/\sigma |y,\rho ,q\right) /\sigma .$

Let $\mathbf{m,\sigma \in \mathbb{R}}^{d}$ and $\mathbf{\rho }\allowbreak
\in \allowbreak (-1,1)^{d-1}$ , $q\in (-1,1]$. Now we are ready to introduce
a multidimensional $q-$Normal distribution $N_{d}\left( \mathbf{m,\sigma }%
^{2}\mathbf{,\rho |}q\right) .$

\begin{definition}
Multidimensional $q-$Normal distribution $N_{d}\left( \mathbf{m,\sigma }^{2}%
\mathbf{,\rho |}q\right) ,$ is the continuous distribution in $\mathbb{R}%
^{d} $ that has density equal to 
\begin{equation*}
g\left( \mathbf{x|m,\sigma }^{2}\mathbf{,\rho ,}q\right) =f_{N}\left( \frac{%
(x_{1}-m_{1}}{\sigma _{1}}|q\right) \prod_{i=1}^{d-1}f_{CN}\left( \frac{%
x_{i+1}-m_{i+1}}{\sigma _{i+1}}|\frac{x_{i}-m_{i}}{\sigma _{i}},\rho
_{i},q\right) /\prod_{i=1}^{d}\sigma _{i}
\end{equation*}%
where $\mathbf{x\allowbreak =\allowbreak }\left( x_{1},\ldots ,x_{d}\right)
^{d},$ $\mathbf{m\allowbreak =\allowbreak }\left( m_{1},\ldots ,m_{d}\right)
,$ $\mathbf{\sigma }^{2}\mathbf{\allowbreak =\allowbreak }\left( \sigma
_{1}^{2},\ldots ,\sigma _{d}^{2}\right) ,$ $\mathbf{\rho \allowbreak
=\allowbreak (}\rho _{1},\ldots ,\rho _{d-1}).$
\end{definition}

As an immediate consequence of the definition we see that $\limfunc{supp}%
(N_{d}(\mathbf{m,\sigma }^{2}\mathbf{|}q))\allowbreak =\mathbf{m+S}\left(
q\right) $. One can also easily see that $\mathbf{m}$ is a shift parameter
and $\mathbf{\sigma }$ is a scale parameter. Hence in particular $\mathbb{E}%
\mathbf{X\allowbreak =\allowbreak m}$. In the sequel we will be mostly
concerned with distributions $N_{d}(\mathbf{0,1,\rho |}q).$

\begin{remark}
Following assertion $1.$ of Proposition \ref{uwaga} we see that distribution 
$N_{d}\left( \mathbf{0,1,0|}q\right) $ is the product distribution of $d$
i.i.d. $q-$Normal distributions. Another words "lack of correlation means
independence"\ in the case of multidimensional $q-$Normal distributions.
More generally if the sequence $\mathbf{\allowbreak \rho \allowbreak
=\allowbreak (}\rho _{1},\ldots ,\rho _{d-1})$ contain, say, $r$ zeros at,
say, positions $t_{1},$ $\ldots ,t_{r}$ then the distribution of $%
N_{d}\left( \mathbf{0},\mathbf{1,\rho }\right) $ is a product distribution
of $r+1$ independent multidimensional $q-$Normal distributions: $%
N_{t_{1}}\left( \mathbf{0,1,(\rho }_{1},\ldots ,\rho _{t_{1}-1})\right)
,\allowbreak \ldots ,\allowbreak N_{d-t_{r}}\left( \mathbf{0,1,(\rho }%
_{t_{r}+1},\ldots ,\rho _{t_{d}}\right) ).$
\end{remark}

Thus in the sequel all considered vectors $\mathbf{\rho }$ will be assumed
to contain only nonzero elements.

Let us introduce the following functions (generating functions of the
families of polynomials):%
\begin{eqnarray}
\varphi \left( x,t|q\right) \allowbreak &=&\allowbreak \sum_{i=0}^{\infty }%
\frac{t^{i}}{\left[ i\right] _{q}!}H_{i}\left( x|q\right) ,  \label{GF_H} \\
\tau \left( x,t|y,\rho ,q\right) \allowbreak &=&\newline
\allowbreak \sum_{i=0}^{\infty }\frac{t^{i}}{\left[ i\right] _{q}!}%
P_{i}\left( x|y,\rho ,q\right) .  \label{GF_ASC}
\end{eqnarray}

The basic properties of the discussed distributions will be collected in the
following Lemma that contains facts from mostly \cite{IA} and the paper \cite%
{bms}.

\begin{lemma}
\label{wlasnosci}i) For $n,m\geq 0:$%
\begin{equation*}
\int_{S\left( q\right) }H_{n}\left( x|q\right) H_{m}\left( x|q\right)
f_{N}\left( x|q\right) dx\allowbreak =\allowbreak \left\{ 
\begin{array}{ccc}
0 & when & n\neq m \\ 
\left[ n\right] _{q}! & when & n=m%
\end{array}%
\right. .
\end{equation*}

ii) For $n\geq 0:$%
\begin{equation*}
\int_{S\left( q\right) }H_{n}\left( x|q\right) f_{CN}\left( x|y,\rho
,q\right) dx=\rho ^{n}H_{n}\left( y|q\right) .
\end{equation*}

iii) For $n,m\geq 0:$%
\begin{equation*}
\int_{S\left( q\right) }P_{n}\left( x|y,\rho ,q\right) P_{m}\left( x|y,\rho
,q\right) f_{CN}\left( x|y,\rho ,q\right) dx\allowbreak =\allowbreak \left\{ 
\begin{array}{ccc}
0 & when & n\neq m \\ 
\left( \rho ^{2}\right) _{n}\left[ n\right] _{q}! & when & n=m%
\end{array}%
\right. .
\end{equation*}

iv) 
\begin{equation*}
\int_{S\left( q\right) }f_{CN}\left( x|y,\rho _{1},q\right) f_{CN}\left(
y|z,\rho _{2},q\right) dy=f_{CN}\left( x|z,\rho _{1}\rho _{2},q\right) .
\end{equation*}

v) For $|t|,\left\vert q\right\vert <1:$%
\begin{equation*}
\sum_{i=0}^{\infty }\frac{W_{i}\left( q\right) t^{i}}{\left( q\right) _{i}}%
\allowbreak =\allowbreak \frac{1}{\left( t\right) _{\infty }^{2}}%
,\sum_{i=0}^{\infty }\frac{W_{i}^{2}\left( q\right) t^{i}}{\left( q\right)
_{i}}\allowbreak =\allowbreak \frac{\left( t^{2}\right) _{\infty }}{\left(
t\right) _{\infty }^{4}},
\end{equation*}%
convergence is absolute, where $W_{i}\left( q\right) $ is defined by (\ref%
{Wn}).

vi) For $(1-q)x^{2}\leq 2$ and $\forall (1-q)t^{2}<1\allowbreak :\allowbreak 
$%
\begin{equation*}
\varphi \left( x,t|q\right) \allowbreak =\allowbreak \prod_{k=0}^{\infty
}\left( 1-\left( 1-q\right) xtq^{k}+\left( 1-q\right) t^{2}q^{2k}\right)
^{-1},
\end{equation*}%
convergence (\ref{GF_H}) is absolute in $t$ \& $x$ and uniform in $x$.
Moreover $\varphi \left( x,t|q\right) $ is positive and $\int_{S\left(
q\right) }\varphi \left( x,t|q\right) f_{N}\left( x|q\right) dx\allowbreak
=\allowbreak 1$. $\allowbreak $ $\varphi \left( t,x|1\right) \allowbreak
=\allowbreak \exp \left( xt-t^{2}/2\right) .$

vii) For $(1-q)\max (x^{2},y^{2})\leq 2,$ $\left\vert \rho \right\vert <1$
and $\forall (1-q)t^{2}<1\allowbreak :\allowbreak $%
\begin{equation*}
\tau \left( x,t|y,\rho ,q\right) \allowbreak =\allowbreak
\prod_{k=0}^{\infty }\frac{\left( 1-\left( 1-q\right) \rho ytq^{k}+\left(
1-q\right) \rho ^{2}t^{2}q^{2k}\right) }{\left( 1-\left( 1-q\right)
xtq^{k}+\left( 1-q\right) t^{2}q^{2k}\right) },
\end{equation*}%
convergence (\ref{GF_ASC}) is absolute in $t$ \& $x$ and uniform in $x$.
Moreover $\tau \left( x,t|\theta ,\rho ,q\right) $ is positive and $%
\int_{S\left( q\right) }\tau \left( x,t|y,\rho ,q\right) f_{CN}\left(
x|y,\rho ,q\right) dx\allowbreak =\allowbreak 1$. \allowbreak $\tau \left(
x,t|y,\rho ,1\right) \allowbreak =\allowbreak \exp \left( t\left( x-\rho
y\right) -t^{2}(1-\rho ^{2})/2\right) .$

viii) For $(1-q)\max (x^{2},y^{2})\leq 2,$ $\left\vert \rho \right\vert <1:$%
\begin{equation}
f_{CN}\left( x|y,\rho ,q\right) \allowbreak \allowbreak =\allowbreak
f_{N}\left( x|q\right) \sum_{n=0}^{\infty }\frac{\rho ^{n}}{[n]_{q}!}%
H_{n}(x|q)H_{n}(y|q)  \label{P-M}
\end{equation}%
and convergence is absolute $t,$ $y$ \& $x$ and uniform in $x$ and $y.$
\end{lemma}

\begin{proof}
i) It is formula 13.1.11 of \cite{IA} with obvious modification for
polynomials $H_{n}$ instead of $h_{n}$ (compare (\ref{q-cont})) and
normalized weight function (i.e. $f_{N})$ ii) Exercise 15.7 of \cite{IA}
also in \cite{bryc1}, iii) Formula 15.1.5 of \cite{IA} with obvious
modification for polynomials $P_{n}$ instead of $p_{n}\left( x|y,\rho
,q\right) \allowbreak =\allowbreak (1-q)^{n/2}P\left( \frac{2x}{\sqrt{1-q}}|%
\frac{2y}{\sqrt{1-q}},\rho ,q\right) $ and normalized weight function (i.e. $%
f_{CN}),$ iv) see (2.6) of \cite{bms}. v) Exercise 12.2(b) and 12.2(c) of 
\cite{IA}. vi)-viii) The exact formulae are known and are given in e.g. \cite%
{IA} (Thm. 13.1.1, 13.1.6) and \cite{AI84} (3.6, 3.10). Absolute convergence
of $\varphi $ and $\tau $ follow (\ref{ogr_H}) and v). Positivity of $%
\varphi $ and $\tau $ follow formulae $1-\left( 1-q\right) xtq^{k}+\left(
1-q\right) t^{2}q^{2k}\allowbreak =\allowbreak
(1-q)(tq^{k}-x/2)^{2}\allowbreak +\allowbreak 1-(1-q)x^{2}/4$ and $1-\left(
1-q\right) \rho ytq^{k}+\left( 1-q\right) \rho ^{2}t^{2}q^{2k}\allowbreak
=\allowbreak (1-q)\rho ^{2}(q^{k}t-y/(2\rho ))^{2}+1-(1-q)y^{2}/4$. Values
of integrals follow (\ref{GF_H}) and (\ref{GF_ASC}) and the fact that $%
\left\{ H_{n}\right\} $ and $\left\{ P_{n}\right\} $ are orthogonal bases in
spaces $L\left( q\right) $ and $CL\left( y,\rho ,q\right) .$
\end{proof}

\begin{corollary}
\label{marginal}Every marginal distribution of multidimensional $q-$Normal
distribution $N_{d}\left( \mathbf{m,\sigma }^{2}\mathbf{,\rho |}q\right) $
is multidimensional $q-$Normal. In particular every one-dimensional
distribution is $q-$Normal. More precisely $i-$th coordinate of $N_{d}\left( 
\mathbf{m,\sigma }^{2}\mathbf{,\rho |}q\right) -$ vector has $(m_{i},\sigma
_{i}^{2},q)-$ Normal distribution.
\end{corollary}

\begin{proof}
By considering transformation $(X_{1},\ldots ,X_{d})\allowbreak
\longrightarrow \allowbreak (\frac{X_{1}-m_{1}}{\sigma _{1}},\ldots ,\frac{%
X_{d}-m_{d}}{\sigma _{d}})$ we reduce considerations to the case $%
N_{d}\left( \mathbf{0,1,\rho |}q\right) $. First let us consider $d-1$
dimensional marginal distributions. The assertion of Corollary is obviously
true since we have assertion iv) of the Lemma \ref{wlasnosci}. We can repeat
this reasoning and deduce that all $d-2,$ $d-3,$ $\ldots ,$ $2$ dimensional
distributions are multidimensional $q-$Normal. The fact that $1-$
dimensional marginal distributions are $q-$normal follows the fact that $%
f_{CN}\left( y|x,\rho ,q\right) $ is a one-dimensional density and
integrates to $1$.
\end{proof}

\begin{corollary}
\label{warunkowanie}If $\mathbf{X\allowbreak \mathbf{\allowbreak
=\allowbreak (}}X_{1},\ldots ,X_{d}\mathbf{)\sim \allowbreak N}_{d}\left( 
\mathbf{m,1,\rho |}q\right) $, then

i) $\forall n\in \mathbb{\mathbb{N}},\allowbreak 1\allowbreak \leq
j_{1}\allowbreak <\allowbreak j_{2}\allowbreak \ldots <\allowbreak
j_{m}\allowbreak <\allowbreak i\leq d:$\newline
$X_{i}|X_{j_{m}},\ldots ,X_{j_{1}}\allowbreak \sim \allowbreak f_{CN}\left(
x_{i}|x_{j_{m}},\prod_{k=j_{m}}^{i-1}\rho _{k},q\right) .$ Thus in
particular \newline
\begin{equation*}
\mathbb{E}\left( H_{n}\left( X_{i}-m_{i}\right) |X_{j_{1}},\ldots
,X_{j_{m}}\right) \allowbreak =\allowbreak \left( \prod_{k=j_{m}}^{i-1}\rho
_{k}\right) ^{n}H_{n}\left( X_{j_{m}}-m_{j_{m}}\right) 
\end{equation*}%
and $\func{var}\left( X_{i}|X_{j_{1}},\ldots ,X_{j_{m}}\right) \allowbreak
=\allowbreak 1-\left( \prod_{k=j_{m}}^{i-1}\rho _{k}\right) ^{2}$.

ii) $\forall n\in \mathbb{\mathbb{N}},1\leq j_{1}<\ldots
j_{k}<i<j_{m}<\ldots <j_{h}\leq d:$%
\begin{equation*}
X_{i}|X_{j_{1}},\ldots X_{j_{k}},X_{j_{m}},\ldots ,X_{j_{h}}\allowbreak \sim
\allowbreak f_{N}\left( x_{i}|q\right) \prod_{l=0}^{\infty }\frac{%
h_{l}\left( x_{j_{k}},x_{j_{m}},\rho _{k}^{\ast }\rho _{m}^{\ast },q\right) 
}{h_{l}(x_{i},x_{j_{k}},\rho _{k}^{\ast },q)h_{l}(x_{i},x_{j_{m}},\rho
_{m}^{\ast },q)},
\end{equation*}%
where $h_{l}\left( x,y,\rho ,q\right) \allowbreak =\allowbreak ((1-\rho
^{2}q^{2l})^{2}-(1-q)\rho q^{l}\left( 1+\rho ^{2}q^{2l}\right) xy\allowbreak
+\allowbreak (1-q)\rho ^{2}q^{2l}(x^{2}+y^{2})),$ $\rho _{k}^{\ast
}\allowbreak =\allowbreak \prod_{i=j_{k}}^{i-1}\rho _{i},$ $\rho _{m}^{\ast
}\allowbreak =\allowbreak \prod_{i=i}^{j_{m}-1}\rho _{i}$. Thus in
particular this density depends only on $X_{j_{k}}$ and $X_{j_{m}}$.
\end{corollary}

\begin{proof}
i) As before, by suitable change of variables we can work with distribution $%
N_{d}\left( \mathbf{0,1,\rho |}q\right) .$ Then following assertion iii) of
the Lemma \ref{wlasnosci} and the fact that $m-$ dimensional marginal, with
respect to which we have to integrate is also multidimensional $q-$Normal
and that the last factor in the product representing density of this
distribution is $f_{CN}\left( x_{i}|x_{j_{m}},\prod_{k=j_{m}}^{i-1}\rho
_{k},q\right) $ we get i).

ii) First of all notice that joint distribution of $(X_{j_{1}},\ldots
X_{j_{k}},X_{i},X_{j_{m}},\ldots ,X_{j_{h}})$ depends only on $%
x_{j_{k}},x_{i},x_{j_{m}}$ since sequence $X_{i},$ \allowbreak $i\allowbreak
=\allowbreak 1\ldots ,n$ is Markov. It is also obvious that the density of
this distribution exist and can be found as a ratio of joint distribution of 
$\left( X_{j_{k}},X_{i},X_{j_{m}}\right) $ divided by the joint density of $%
\left( X_{j_{k}},X_{j_{m}}\right) .$ Keeping in mind that $%
X_{j_{k}},X_{i},X_{j_{m}}$ have the same marginal $f_{N}$ and because of
assertion iv of Lemma \ref{wlasnosci} we get the postulated form.
\end{proof}

Having Lemma \ref{wlasnosci} we can present Proposition concerning mutual
relationship between spaces $L\left( q\right) $ and $CL\left( y,\rho
,q\right) $ defined at the beginning of previous section.

\begin{proposition}
\label{przestrzenie}$\forall q\in (-1,1),\allowbreak y\allowbreak \in
\allowbreak S\left( q\right) ,\allowbreak \left\vert \rho \right\vert
<1:L\left( q\right) \allowbreak =\allowbreak CL\left( y,\rho ,q\right) $.
Besides $\exists C_{1}\left( y,\rho ,q\right) ,$ $C_{2}\left( y,\rho
,q\right) :$ $\left\Vert g\right\Vert _{L}\leq C_{1}\left\Vert g\right\Vert
_{CL}$ and $\left\Vert g\right\Vert _{CL}\leq C_{2}\left\Vert g\right\Vert
_{L}$ for every $g\in L\left( q\right) .$
\end{proposition}

\begin{proof}
Firstly observe that : $(1-\rho ^{2}q^{2k})^{2}\allowbreak -\allowbreak
(1-q)\rho q^{k}(1+\rho ^{2}q^{2k})xy\allowbreak +\allowbreak (1-q)\rho
^{2}(x^{2}+y^{2})q^{2k}\allowbreak =\allowbreak (1-q)\rho ^{2}q^{2k}\left(
x-y\frac{\rho q^{k}+\rho ^{-1}q^{-k}}{2}\right) ^{2}\allowbreak +\allowbreak
(1-(1-q)y^{2}/4)(1-\rho ^{2}q^{2k})^{2}$ which is elementary to prove. We
will use modification of the formula (\ref{P-M}) that is obtained from it by
dividing both sides by $f_{N}\left( x|q\right) .$ That is formula:%
\begin{eqnarray*}
&&\prod_{k=0}^{\infty }\frac{\left( 1-\rho ^{2}q^{k}\right) }{(1-\rho
^{2}q^{2k})^{2}-(1-q)\rho q^{k}(1+\rho ^{2}q^{2k})xy+(1-q)\rho
^{2}(x^{2}+y^{2})q^{2k}} \\
&=&\sum_{n=0}^{\infty }\frac{\rho ^{n}}{\left[ n\right] _{q}!}H_{n}\left(
x|q\right) H_{n}\left( y|q\right) .
\end{eqnarray*}%
Now we use (\ref{ogr_H}) and assertion v) of Lemma \ref{wlasnosci} and get $%
\forall x,y\in S\left( q\right) :$ 
\begin{equation*}
f_{CN}\left( x|y,\rho ,q\right) \leq f_{N}\left( x|q\right) \frac{\left(
\rho ^{2}\right) _{\infty }}{\left( \rho \right) _{\infty }^{4}}.
\end{equation*}%
Hence $C_{2}\allowbreak =\frac{\left( \rho ^{2}\right) _{\infty }}{\left(
\rho \right) _{\infty }^{4}}\allowbreak $ and $\left\Vert g\right\Vert
_{CL}^{2}\allowbreak \leq \allowbreak \left\Vert g\right\Vert _{L}^{2},$ for
every $g\in L\left( q\right) .$ Thus $g\in CL\left( y,\rho ,q\right) $.

Conversely to take a function $g\in CL\left( y,\rho ,q\right) $. We have 
\begin{equation*}
\infty \allowbreak >\allowbreak \int_{S\left( q\right) }\left\vert g\left(
x\right) \right\vert ^{2}f_{CN}\left( x|y,\rho ,q\right) dx.
\end{equation*}%
Now we keeping in mind that $(1-\rho ^{2}q^{2k})^{2}\allowbreak -\allowbreak
(1-q)\rho q^{k}(1+\rho ^{2}q^{2k})xy\allowbreak +\allowbreak (1-q)\rho
^{2}(x^{2}+y^{2})q^{2k}\allowbreak $ is a quadratic function in $x,$ we
deduce that it reaches its maximum for $x\in S\left( q\right) $ on the end
points of $S\left( q\right) $. Hence we have 
\begin{eqnarray*}
&&(1-\rho ^{2}q^{2k})^{2}\allowbreak -\allowbreak (1-q)\rho q^{k}(1+\rho
^{2}q^{2k})xy\allowbreak +\allowbreak (1-q)\rho
^{2}(x^{2}+y^{2})q^{2k}\allowbreak \\
&\leq &\allowbreak (1+\rho ^{2}q^{2k}\allowbreak +\allowbreak \sqrt{1-q}%
\left\vert \rho y\right\vert \left\vert q\right\vert ^{k})^{2}.
\end{eqnarray*}%
Since for $\forall y\in S\left( q\right) ,\allowbreak \left\vert \rho
\right\vert ,\allowbreak \left\vert q\right\vert <1:$%
\begin{equation*}
\prod_{k=0}^{\infty }(1+\rho ^{2}q^{2k}\allowbreak +\allowbreak \sqrt{1-q}%
\left\vert \rho y\right\vert \left\vert q\right\vert ^{k})^{2}\allowbreak
<\allowbreak \infty
\end{equation*}%
and we see that 
\begin{gather*}
\infty \allowbreak >\allowbreak \int_{S\left( q\right) }\left\vert g\left(
x\right) \right\vert ^{2}f_{CN}\left( x|y,\rho ,q\right) dx\allowbreak
\allowbreak =\int_{S\left( q\right) }\left\vert g\left( x\right) \right\vert
^{2}f_{N}\left( x|q\right) \\
\allowbreak \times \prod_{k=0}^{\infty }\frac{1-\rho ^{2}q^{k}}{(1-\rho
^{2}q^{2k})^{2}-(1-q)\rho q^{k}(1+\rho ^{2}q^{2k})xy+(1-q)\rho
^{2}(x^{2}+y^{2})q^{2k}}dx\allowbreak \\
\geq \frac{\left( \rho ^{2}\right) _{\infty }}{\prod_{k=0}^{\infty }(1+\rho
^{2}q^{2k}\allowbreak +\allowbreak \sqrt{1-q}\left\vert \rho y\right\vert
\left\vert q\right\vert ^{k})^{2}}\allowbreak \int_{S\left( q\right)
}\left\vert g\left( x\right) \right\vert ^{2}f_{N}\left( x|q\right) dx.
\end{gather*}%
$\allowbreak \allowbreak \allowbreak $ So $g\in L\left( q\right) .$
\end{proof}

\begin{remark}
Notice that the assertion of Proposition \ref{przestrzenie} is not true for $%
q\allowbreak =\allowbreak 1$ since then the respective densities are $%
N\left( 0,1\right) $ and $N\left( \rho y,1-\rho ^{2}\right) $.
\end{remark}

\begin{remark}
Using assertion of Proposition \ref{przestrzenie} we can rephrase Corollary %
\ref{warunkowanie} in terms of contraction $\mathcal{R}\left( \rho ,q\right)
,$ (defined by (\ref{kontrakcja}), below). For $g\in L\left( q\right) $ we
have 
\begin{equation*}
\mathbb{E}\left( g\left( X_{i}\right) |X_{j_{1}},\ldots ,X_{j_{m}}\right)
\allowbreak =\allowbreak \mathcal{R}\left( \prod_{k=j_{m}}^{i-1}\rho
_{k},q\right) \left( g\left( X_{j_{m}}\right) \right) ,
\end{equation*}%
where $\mathcal{R}\left( \rho ,q\right) $ is a contraction on the space $%
L\left( q\right) $ defined by the formula (using polynomials $H_{n}$ for $%
\left\vert \rho \right\vert ,\left\vert q\right\vert \allowbreak
<\allowbreak 1)$ : 
\begin{equation}
L\left( q\right) \allowbreak \ni \allowbreak f\allowbreak =\allowbreak
\sum_{i=0}^{\infty }a_{i}H_{i}\left( x|q\right) \allowbreak \longrightarrow 
\mathcal{R}\left( \rho ,q\right) \left( f\right) \allowbreak =\allowbreak
\sum_{i=0}^{\infty }a_{i}\rho ^{i}H_{i}\left( x|q\right) .
\label{kontrakcja}
\end{equation}%
By the way it is known that $\mathcal{R}$ is not only contraction but also
ultra contraction i.e. mapping $L_{2}$ on $L_{\infty }$ (Bo\.{z}ejko).
\end{remark}

We have also the following almost obvious observation that follows, in fact,
from assertion iii) of the Lemma \ref{wlasnosci}.

\begin{proposition}
\label{regresja}Suppose that $\mathbf{X\allowbreak \mathbf{\allowbreak
=\allowbreak }}(X_{1},\ldots ,X_{d})\allowbreak \mathbf{\sim \allowbreak N}%
_{d}\left( 0\mathbf{,1,\rho |}q\right) $ and $g\in L\left( q\right) $.
Assume that for some $n\in \mathbb{\mathbb{N}},\allowbreak $ and $%
1\allowbreak \leq j_{1}\allowbreak <\allowbreak j_{2}\allowbreak \ldots
<\allowbreak j_{m}\allowbreak <\allowbreak i\leq d.$

i) If $\mathbb{E}\left( g\left( X_{i}\right) |X_{j_{1}},\ldots
,X_{j_{m}}\right) \allowbreak =\allowbreak $polynomial of degree at most $n$
of $X_{j_{m}},$ then function $g$ must be also a polynomial of degree at
most $n.$

ii) If additionally $\mathbb{E}g\left( X_{i}\right) \allowbreak =\allowbreak
0$ 
\begin{equation}
\mathbb{E((\mathbb{E(}}g\mathbb{\mathbb{(}}X_{i})|X_{j_{1}},\ldots
,X_{j_{m}})^{2})\leq r^{2}\mathbb{E}g^{2}(X_{i}), 
\tag{Generalized
Gebelein's inequality}
\end{equation}%
where $r\allowbreak =\allowbreak \prod_{k=j_{m}}^{i-1}\rho _{k}.$
\end{proposition}

\begin{proof}
i) The fact that $\mathbb{E}\left( g\left( X_{i}\right) |X_{j_{1}},\ldots
,X_{j_{m}}\right) $ is a function of $X_{j_{m}}$ only, is obvious. Since $%
g\in L\left( q\right) $ we can expand it in the series $g\left( x\right)
\allowbreak =\allowbreak \sum_{i\geq 0}c_{i}H_{i}\left( x|q\right) $. By
Corollary \ref{warunkowanie} we know that $\mathbb{E}\left( g\left(
X_{i}\right) |X_{j_{1}},\ldots ,X_{j_{m}}\right) \allowbreak =\allowbreak
\sum_{i\geq 0}c_{i}r^{i}H\left( X_{j_{m}}|q\right) $ for $r\allowbreak
=\allowbreak \prod_{k=j_{m}}^{i-1}\rho _{k}$. Now since $c_{i}r^{i}%
\allowbreak =\allowbreak 0$ for $i>n$ and $r\neq 0$ we deduce that $%
c_{i}\allowbreak =\allowbreak 0$ for $i>n.$

ii) Suppose $g\left( x\right) \allowbreak =\allowbreak \sum_{i=1}^{\infty
}g_{i}H_{i}(x)$. We have $\mathbb{\mathbb{E(}}g\mathbb{\mathbb{(}}%
X_{i})|X_{j_{1}},\ldots ,X_{j_{m}})\allowbreak =\allowbreak
\sum_{i=1}^{\infty }g_{i}r^{i}H_{i}(Y_{j_{m}})$. Hence $\mathbb{E((\mathbb{E(%
}}g\mathbb{\mathbb{(}}X_{i})|X_{j_{1}},\ldots ,X_{j_{m}})^{2})\allowbreak
=\allowbreak \sum_{i=1}^{\infty }g_{i}^{2}r^{2i}\left[ i\right] _{q}!\leq
r^{2}\sum_{i=1}^{\infty }g_{i}^{2}\left[ i\right] _{q}!\allowbreak
=\allowbreak r^{2}\mathbb{E}g^{2}(X_{i}).$
\end{proof}

\begin{remark}
As it follows from the above mentioned definition, the multidimensional $q-$%
Normal distribution is not a true generalization of $n-$dimensional Normal
law $\mathbf{N}_{n}\left( \mathbf{m,\Sigma }\right) $. It a generalization
of distribution $\mathbf{N}_{n}\left( \mathbf{m,\Sigma }\right) $ with very
specific matrix $\mathbf{\Sigma }$ namely with entries equal to $\sigma
_{ii}\allowbreak \allowbreak =\allowbreak \sigma _{i}^{2};\allowbreak \sigma
_{ij}\allowbreak =\allowbreak \sigma _{i}\sigma _{j}\prod_{k=i}^{j-1}\rho
_{k}$ for $i<j$ and $\sigma _{ij}\allowbreak =\allowbreak \sigma _{ji}$ for $%
i>j$ where $\sigma _{i}$ ; $i\allowbreak =\allowbreak 1,\ldots ,n$ are some
positive numbers and $\left\vert \rho _{i}\right\vert <1,$ $i\allowbreak
=\allowbreak 1,\ldots ,n-1.$
\end{remark}

\begin{proof}
Follows the fact that two dimensional $q-$Normal distribution of say $\left(
X_{i}/\sigma _{i},X_{j}/\sigma _{j}\right) $ has density $f_{N}\left(
x_{i}|q\right) f_{CN}\left( x_{j}|x_{i},\prod_{k=i}^{j-1}\rho _{k},q\right) $
if $i<j.$
\end{proof}

\begin{remark}
Suppose that $(X_{1},\ldots ,X_{n})\allowbreak \sim \allowbreak N_{n}\left( 
\mathbf{m,\sigma ,\rho |}q\right) $ then $X_{1},\ldots ,X_{n}$ form a finite
Markov chain with $X_{i}\allowbreak \sim \allowbreak $ $(m_{i},\sigma
_{i}^{2},q)-$Normal and transition density $X_{i}|X_{i-1}=y\allowbreak \sim
\allowbreak (m_{i},\sigma _{i}^{2},y,\rho _{i-1},q)$-Conditional Normal
distribution
\end{remark}

Following assertions vi) and vii) of Lemma \ref{wlasnosci} we deduce that
for $\forall t^{2}\mathbb{<}1/(1-q),$ $\left\vert \rho \right\vert <1$
functions $\varphi \left( x,t|q\right) f_{N}\left( x|q\right) \allowbreak $
and $\tau \left( x,t|y,\rho ,q\right) f_{CN}\left( x|y,\rho ,q\right)
\allowbreak $ are densities. Hence we obtain new densities with additional
parameter $t$. This observation leads to the following definitions:

\begin{definition}
Let $\left\vert q\right\vert \in (-1,1],t^{2}<1/(1-q),x\in S\left( q\right) $%
. A distribution with the density $\varphi \left( x,t|q\right) f_{N}\left(
x|q\right) $ will be called \emph{modified }$\left( t,q\right) -$\emph{%
Normal (briefly }$\left( t,q\right) -$MN distribution).
\end{definition}

We have immediate observation that follows from assertion vi) of Lemma \ref%
{wlasnosci}.

\begin{proposition}
\label{MN}i) $\int_{S\left( q\right) }\varphi \left( y,t|q\right)
f_{N}\left( y|q\right) f_{CN}\left( x|y,\rho ,q\right) dy\allowbreak
=\allowbreak \varphi \left( x,t\rho |q\right) f_{N}(x|q)$

ii) Let $X\allowbreak \sim \allowbreak \left( t,q\right) -$MN. Then for $%
n\in \mathbb{N}:\mathbb{E}\left( H_{n}\left( X|q\right) \right) \allowbreak
=\allowbreak t^{n}.$
\end{proposition}

\begin{proof}
i) Using assertions vi) and viii) of the Lemma \ref{wlasnosci} we get: 
\newline
\begin{eqnarray*}
&&\int_{S\left( q\right) }\varphi \left( y,t|q\right) f_{N}\left( y|q\right)
f_{CN}\left( x|y,\rho ,q\right) dy\allowbreak \\
&=&f_{N}\left( x|q\right) \int_{S\left( q\right) }f_{N}\left( y|q\right)
\sum_{i=1}^{\infty }\frac{t^{i}}{[i]_{q}!}H_{i}\left( y|q\right)
\sum_{j=0}^{\infty }\frac{\rho ^{j}}{[j]_{q}!}H_{j}\left( y|q\right)
H_{j}\left( x|q\right) dy.
\end{eqnarray*}%
\newline
Now utilizing assertion ii) of the same Lemma we get: 
\begin{eqnarray*}
&&\int_{S\left( q\right) }\varphi \left( y,t|q\right) f_{N}\left( y|q\right)
f_{CN}\left( x|y,\rho ,q\right) dy\allowbreak \\
&=&\allowbreak f_{N}\left( x|q\right) \sum_{i=0}^{\infty }\frac{\left( t\rho
\right) ^{i}}{[i]_{q}!}H_{i}\left( x|q\right) \allowbreak =\allowbreak
\varphi \left( x,t\rho |q\right) f_{N}(x|q).
\end{eqnarray*}
To get ii) we utilize assertion vi) of the Lemma \ref{wlasnosci}.
\end{proof}

In particular we have:

\begin{corollary}
If $X\allowbreak \sim \allowbreak (t,q)-MN$, then $\mathbb{E}X\allowbreak
=\allowbreak t,$ $\limfunc{var}\left( X\right) \allowbreak =\allowbreak 1,$ $%
\mathbb{E}\left( \left( X-t\right) ^{3}\right) \allowbreak =\allowbreak
-t\left( 1-q\right) ,$ $\mathbb{E}\left( X-t\right) ^{4}\allowbreak
=\allowbreak 2+q-t^{2}\left( 5+6q+q^{2}\right) $.
\end{corollary}

\begin{proof}
We have $x^{3}\allowbreak =\allowbreak H_{3}\left( x|q\right) +\left(
2+q\right) H_{1}\left( x\right) $ and $H_{4}\left( x|q\right)
+(3+2q+q^{2})H_{2}\left( x|q\right) +2+q\allowbreak =\allowbreak x^{4}$ so $%
\mathbb{E}\left( X-t\right) ^{3}\allowbreak =\allowbreak \mathbb{E}\left(
X^{3}\right) -3\mathbb{E}\left( X^{2}\right) t+3\mathbb{E}\left( X\right)
t^{2}-t^{3}\allowbreak =\allowbreak t^{3}+\left( 2+q\right) t\allowbreak
-\allowbreak 3t\left( 1+t^{2}\right) +3t^{3}-t^{3}\allowbreak =\allowbreak
-t\left( 1-q\right) $ and $\mathbb{E}\left( X-t\right) ^{4}\allowbreak
=\allowbreak \mathbb{E}\left( X^{4}\right) -4t\mathbb{E}\left( X^{3}\right)
+6t^{2}\mathbb{E}\left( X^{2}\right) -4t^{3}\mathbb{E}\left( X\right)
+t^{4}\allowbreak =\allowbreak t^{4}-(3+2q+q^{2})t^{2}+2+q-4t\left(
t^{3}+\left( 2+q\right) t\right) +6t^{2}\left( t^{2}+1\right) -4t^{4}+t^{4}$
which reduces to $2+q-t^{2}\left( 5+6q+q^{2}\right) .$
\end{proof}

In particular kurtosis of $\left( t,q\right) -MN$ distributions is equal to $%
-\left( 1-q\right) -t^{2}\left( 5+6q+q^{2}\right) $. Hence it is negative
and less, for $t\neq 0,$ than that of $q-$Normal which is also negative
(equal to $-\left( 1-q\right) $).

Assertion i) of the Proposition \ref{MN} leads to the generalization of the
multidimensional $q-$Normal distribution that allows different
one-dimensional and other marginals.

\begin{definition}
A distribution in $\mathbb{R}^{d}$ having density equal to 
\begin{equation*}
\varphi \left( x_{1},t|q\right) f_{N}\left( x_{1}|q\right)
\prod_{i=1}^{d-1}f_{CN}\left( x_{i+1}|x_{i},\rho _{i},q\right) ,
\end{equation*}%
where $x_{i}\in S\left( q\right) ,$ $\rho _{i}\in (-1,1)\backslash \left\{
0\right\} ,$ $i\allowbreak =\allowbreak 1,\ldots ,d-1,$ $\left\vert
q\right\vert \in (-1,1],$ $t^{2}<1/(1-q)$ will be called \emph{modified
multidimensional }$\emph{q-}$\emph{Normal distribution }(briefly\emph{\ }$%
MMN_{d}(\mathbf{\rho |}q,t)$).\emph{\ }
\end{definition}

Reasoning in the similar way as in the proof of Corollary \ref{marginal} and
utilizing observation following from Proposition \ref{MN}, we have
immediately the following observation.

\begin{proposition}
Let $(X_{1},\ldots ,X_{d})\allowbreak \sim \allowbreak MMN_{d}(\mathbf{\rho |%
}q,t)$. Then every marginal of it is also modified multidimensional $q-$%
Normal. In particular $\forall i\allowbreak =\allowbreak 1,\ldots ,d$ $:$ $%
X_{i}\allowbreak \sim \allowbreak \left( t\prod_{k=1}^{i-1}\rho
_{k},q\right) -MN$
\end{proposition}

\begin{remark}
Suppose that $(X_{1},\ldots X_{d})\allowbreak \sim \allowbreak MMN_{d}(%
\mathbf{\rho |}q,t)$ and define $\rho _{0}=1,$ then the sequence $%
X_{1},\ldots X_{d}$ form a non-stationary Markov chain such that $%
X_{i}\allowbreak \sim \allowbreak \varphi \left( x|t\prod_{k=1}^{i-1}\rho
_{k},q\right) f_{N}\left( x|q\right) $ with transitional probability $%
X_{i+1}|X_{i}=y\allowbreak \sim \allowbreak f_{CN}\left( x|y,\rho
_{i},q\right) .$
\end{remark}

We can define another one-dimensional distribution depending on $4$
parameters. We have:

\begin{definition}
Let $\left\vert q\right\vert \in (-1,1],t^{2}<1/(1-q),x,y\in S\left(
q\right) ,$ $\left\vert \rho \right\vert <1$. A distribution with the
density $\tau \left( x,t|y,\rho ,q\right) f_{CN}\left( x|y,\rho ,q\right) $
will be called \emph{modified }$\left( y,\rho ,t,q\right) -$\emph{%
Conditional Normal }(briefly $\left( y,\rho ,t,q\right) -$MCN)\emph{.}
\end{definition}

We have immediate observation that follows from assertion vii) of Lemma \ref%
{wlasnosci}.

\begin{proposition}
Let $X\allowbreak \sim \allowbreak \left( y,\rho ,t,q\right) -$MCN. Then for 
$n\in \mathbb{N}:\mathbb{E}\left( P_{n}\left( X|y,\rho ,q\right) \right)
\allowbreak =\allowbreak \left( \rho ^{2}\right) _{n}t^{n}.$
\end{proposition}

Hence in particular one can state the following Corollary.

\begin{corollary}
$\mathbb{E}X\allowbreak =\allowbreak \rho y\allowbreak +\allowbreak (1-\rho
^{2})t,$ $\limfunc{var}\left( X\right) \allowbreak =\allowbreak (1-\rho
^{2})(1-(1-q)ty\rho +(1-q)t^{2}\rho ^{2}).$
\end{corollary}

\begin{proof}
Follows expressions for first two Al-Salam-Chihara polynomials. Namely we
have: $P_{1}(x|y,\rho ,q)\allowbreak =\allowbreak x-\rho y,$ $P_{2}(x|y,\rho
,q)\allowbreak =\allowbreak x^{2}-1\allowbreak +\allowbreak \rho
^{2}\allowbreak +\allowbreak q\rho ^{2}y^{2}-\allowbreak x\rho y(1+q).$
\end{proof}

We can define two formulae for densities of multidimensional distributions
in $\mathbb{R}^{d}$. Namely one of them would have density of the form%
\begin{equation*}
\varphi \left( x_{1},t|q\right) f_{N}\left( x_{1}|q\right) \tau \left(
x_{2},t|x_{1},\rho _{1},q\right) \prod_{i=1}^{d-1}f_{CN}\left(
x_{i+1}|x_{i},\rho _{i},q\right)
\end{equation*}%
and the other of the form%
\begin{equation*}
f_{N}\left( x_{1}|q\right) \tau \left( x_{2},t|x_{1},\rho _{1},q\right)
\prod_{i=1}^{d-1}f_{CN}\left( x_{i+1}|x_{i},\rho _{i},q\right) .
\end{equation*}%
However to find marginals of such families of distributions is a challenge
and an open question. In particular are they also of modified conditional
normal type?

\section{Main Results}

In this section we are going to study properties of $3$ dimensional case of
multidimensional normal distribution. To simplify notation we will consider
vector $(Y,X,Z)$ having distribution $\mathbf{N}_{3}((0,0,0),(1,1,1),(\rho
_{1},\rho _{2})|q)$ that is having density $f_{CN}\left( y|x,\rho
_{1},q\right) f_{CN}(x|z,\rho _{2},q)f_{N}(z|q)$. We start with the
following obvious result:

\begin{remark}
Conditional distribution $X|Y,Z$ has density 
\begin{equation}
\phi \left( x|y,z,\rho _{1},\rho _{2},q\right) \allowbreak =\allowbreak
f_{N}\left( x|q\right) \frac{\left( \rho _{1}^{2},\rho _{2}^{2}\right)
_{\infty }}{\left( \rho _{1}^{2}\rho _{2}^{2}\right) _{\infty }}%
\prod_{i=0}^{\infty }\frac{w_{k}\left( y,z,\rho _{1}\rho _{2},q\right) }{%
w_{k}\left( x,y,\rho _{1},q\right) w_{k}\left( x,z,\rho _{2},q\right) },
\label{_x|yz}
\end{equation}%
where we denoted $w_{k}\left( s,t,\rho ,q\right) \allowbreak =\allowbreak
(1-\rho ^{2}q^{2k})^{2}-(1-q)\rho q^{k}(1+\rho ^{2}q^{2k})st+(1-q)\rho
^{2}(s^{2}+t^{2})q^{2k}.$
\end{remark}

\begin{proof}
It is in fact rewritten version of the proof of assertion ii) of Corollary %
\ref{warunkowanie}.
\end{proof}

\begin{remark}
Notice that $\phi \left( x|y,z,\rho _{1},\rho _{2},q\right) $ is the
re-scaled Askey-Wilson density. Namely $\phi \left( x|y,z,\rho _{1},\rho
_{2},q\right) \allowbreak =\allowbreak \psi (\frac{\sqrt{1-q}}{2}x|a,b,c,d)$
where 
\begin{eqnarray*}
a\allowbreak &=&\allowbreak \frac{\sqrt{1-q}}{2}\rho _{1}(y\allowbreak
-\allowbreak i\sqrt{\frac{4}{1-q}-y^{2}}),b\allowbreak =\allowbreak \frac{%
\sqrt{1-q}}{2}\rho _{1}(y\allowbreak +\allowbreak i\sqrt{\frac{4}{1-q}-y^{2}}%
), \\
c\mathbb{\allowbreak } &\mathbb{=\allowbreak }&\frac{\sqrt{1-q}}{2}\rho
_{1}(z\allowbreak -\allowbreak i\sqrt{\frac{4}{1-q}-z^{2}}),d\allowbreak
=\allowbreak \frac{\sqrt{1-q}}{2}\rho _{1}(z\allowbreak +\allowbreak i\sqrt{%
\frac{4}{1-q}-z^{2}}).
\end{eqnarray*}
and $\psi (t|a,b,c,d)$ is a normalized (that is multiplied by a constant so
that its integral is $1)$ weight function of Askey-Wilson polynomials.
Compare e.g. \cite{AW85} and \cite{IA}. Hence our results would concern
properties of Askey - Wilson density and Askey-Wilson polynomials.
\end{remark}

Let us denote 
\begin{equation*}
g_{n}(y,z,\rho _{1},\rho _{2},q)\allowbreak =\allowbreak \int_{S\left(
q\right) }H_{n}\left( x|q\right) \phi \left( x|y,z,\rho _{1},\rho
_{2},q\right) .
\end{equation*}
Our main result is the following

\begin{theorem}
\label{main}i) $\phi \left( x|y,z,\rho _{1},\rho _{2}|q\right) \allowbreak
=\allowbreak f_{N}\left( x|q\right) \allowbreak \times \allowbreak $\newline
$\sum_{n=0}^{\infty }\frac{H_{n}\left( x|q\right) }{\left[ n\right] _{q}!}%
g_{n}\left( y,z,\rho _{1},\rho _{2},q\right) ,$ where $g_{n}$ is a
polynomial of order $n$ in $(y,z).$

ii) More over polynomial $g_{n}$ has the following structure $\forall n\geq
1:$ $g_{n}\left( y,z,\rho _{1},\rho _{2},q\right) \allowbreak =\allowbreak $ 
$\sum_{i=0}^{n}\QATOPD[ ] {n}{i}_{q}\rho _{1}^{i}\rho _{2}^{n-i}\Theta
_{i,n-i}\left( y,z,\rho _{1}\rho _{2}|q\right) ,$ where $\Theta _{k,l}\left(
y,z,\rho _{1}\rho _{2}|q\right) $ is a polynomial in $y$ of order $k$ and in 
$z$ of order $l$. Moreover $\Theta _{0,n}\left( y,z,0|q\right) \allowbreak
=\allowbreak H_{n}\left( z|q\right) $ and $\Theta _{n,0}\left(
y,z,0,q\right) \allowbreak =\allowbreak H_{n}\left( y|q\right) .$
\end{theorem}

\begin{remark}
Assertion i) of the Theorem \ref{main} is in fact a generalization of
Poisson-Mehler formula (that is assertion viii) of Lemma \ref{wlasnosci})
for Askey-Wilson density.
\end{remark}

\begin{remark}
Notice also that for $q\allowbreak =\allowbreak 1,$ $\phi $ is a density
function of normal distribution $N\left( \frac{y\rho _{1}\left( 1-\rho
_{2}^{2}\right) +z\rho _{2}\left( 1-\rho _{1}^{2}\right) }{1-\rho
_{1}^{2}\rho _{2}^{2}},\frac{\left( 1-\rho _{1}^{2}\right) \left( 1-\rho
_{2}^{2}\right) }{1-\rho _{1}^{2}\rho _{2}^{2}}\right) $ and it is obvious
that expectation of any polynomial is a polynomial in $\frac{y\rho
_{1}\left( 1-\rho _{2}^{2}\right) +z\rho _{2}\left( 1-\rho _{1}^{2}\right) }{%
1-\rho _{1}^{2}\rho _{2}^{2}}$. Hence it turns out that this is true for all 
$q-$Normal distributions for $q\in (-1,1].$
\end{remark}

As a Corollary we have the following result.

\begin{corollary}
\label{postac}Let $\mathbf{X\allowbreak \mathbf{\allowbreak =\allowbreak (}}%
X_{1},\ldots ,X_{d}\mathbf{)\sim \allowbreak N}_{d}\left( 0\mathbf{,1,\rho |}%
q\right) $. Let us select indices $1\leq j_{1}<\ldots j_{k}<i<j_{m}<\ldots
<j_{h}\leq d$. Then%
\begin{eqnarray}
\forall n &\in &\mathbb{N}:\mathbb{E(}H_{n}\left( X_{i}|q\right)
|X_{j_{1}},\ldots ,X_{j_{k}},X_{j_{m}},\ldots ,X_{j_{h}})=  \label{wzor} \\
&&\sum_{r=0}^{\left\lfloor n/2\right\rfloor
}\sum_{l=0}^{n-2r}A_{r,-\left\lfloor n/2\right\rfloor +r+l}^{\left( n\right)
}H_{l}\left( X_{j_{k}}|q\right) H_{n-2r-l}\left( X_{j_{m}}|q\right) ,  \notag
\end{eqnarray}%
\newline
for $\left\lfloor \frac{n+2}{2}\right\rfloor \left\lfloor \frac{n+3}{2}%
\right\rfloor $ constants (depending only on $n,$ $q,\mathbf{\rho }$ and
numbers $i,$ $j_{k}$ $j_{m})$ $A_{r,s}^{(n)};$. $r\allowbreak =\allowbreak
0,\ldots ,\left\lfloor n/2\right\rfloor ,$ $s\allowbreak =\allowbreak
-\left\lfloor n/2\right\rfloor +r,\ldots ,-\left\lfloor n/2\right\rfloor
+r\allowbreak +\allowbreak n-2r.$
\end{corollary}

Corollary bellow gives detailed form of coefficients $A_{r,s}^{(n)}$ for $%
n\allowbreak =\allowbreak 1,\ldots ,4.$

\begin{corollary}
\label{_3i4}Let $\mathbf{X\allowbreak \mathbf{\allowbreak =\allowbreak (}}%
X_{1},\ldots ,X_{d}\mathbf{)\sim \allowbreak N}_{d}\left( 0\mathbf{,1,\rho |}%
q\right) $. Let $1\leq i-1<i<i+1\leq d$. Then:

$\mathbb{E}\left( H_{n}\left( X_{i}\right) |\mathcal{F}_{\neq i}\right)
\allowbreak =\allowbreak \sum_{r=0}^{\left\lfloor n/2\right\rfloor
}\sum_{l=0}^{n-2r}A_{r,-\left\lfloor n/2\right\rfloor +r+l}^{\left( n\right)
}H_{l}\left( X_{i-1}|q\right) H_{n-2r-l}\left( X_{i+1}|q\right) ,$

where $A_{0,-\left\lfloor n/2\right\rfloor +l}^{(n)}\allowbreak =\allowbreak 
\QATOPD[ ] {n}{l}_{q}\frac{\rho _{i-1}^{n-l}\left( \rho _{i}^{2}\right)
_{n-l}\rho _{i}^{l}\left( \rho _{i-1}^{2}\right) _{l}}{\left( \rho
_{i-1}^{2}\rho _{i}^{2}\right) _{n}},$ $l\allowbreak =\allowbreak 0,\ldots
,n,$ $n\allowbreak =\allowbreak 1\ldots ,4$. If $n\allowbreak \leq 3$ then $%
A_{1,-\left\lfloor n/2\right\rfloor +l}^{\left( n\right) }\allowbreak
=\allowbreak -\left[ n-1\right] _{q}\rho _{i-1}\rho _{i}A_{0,-\left\lfloor
n/2\right\rfloor +l}^{\left( n\right) },$ $l\allowbreak =\allowbreak
1,\ldots ,n-1,$ If $n\allowbreak =\allowbreak 4$ then $A_{1,j}^{\left(
4\right) }\allowbreak =\allowbreak -\left[ 3\right] _{q}\rho _{i-1}\rho
_{i}A_{0,j}^{(4)},$ $j\allowbreak =\allowbreak -1,1$ and $%
A_{1,0}^{(4)}\allowbreak =\allowbreak -\left[ 2\right] _{q}^{2}\rho
_{i-1}\rho _{i}A_{0,0}^{(4)},$ $A_{2,0}^{(4)}\allowbreak =\allowbreak
q(1+q)\rho _{i-1}^{2}\rho _{i}^{2}A_{0,0}^{(4)}.$

In particular : 
\begin{gather}
\func{var}\left( X_{i}|X_{1},\ldots ,X_{i-1},X_{i+1},\ldots ,X_{d}\right)
\allowbreak =\allowbreak \frac{(1-\rho _{i-1}^{2})\left( 1-\rho
_{i}^{2}\right) }{\left( 1-q\rho _{i-1}^{2}\rho _{i}^{2}\right) }\allowbreak
\label{warvar1} \\
\times \allowbreak (1-\frac{\left( 1-q\right) (X_{i-1}-\rho _{i-1}\rho
_{i}X_{i+1})\left( X_{i}-X_{i-1}\rho _{i-1}\rho _{i}\right) }{\left( 1-\rho
_{i-1}^{2}\rho _{i}^{2}\right) ^{2}}).  \label{warvar2}
\end{gather}
\end{corollary}

\begin{remark}
Notice that in general conditional variance \newline
$\func{var}\left( X_{i}|X_{1},\ldots ,X_{i-1},X_{i+1},\ldots ,X_{d}\right) $
is not nonrandom indicating that $q-$Normal distribution does not behave as
Normal in this case, however if we set $q=1$ in (\ref{warvar1}) then we get $%
\func{var}\left( X_{i}|X_{1},\ldots ,X_{i-1},X_{i+1},\ldots ,X_{d}\right)
\allowbreak =\allowbreak \frac{(1-\rho _{i-1}^{2})\left( 1-\rho
_{i}^{2}\right) }{\left( 1-\rho _{i}^{2}\rho _{i-1}^{2}\right) }$ as it
should be in the normal case.
\end{remark}

Notice that examining the form of coefficients $A_{m,k}^{(n)}$ for $%
n\allowbreak =\allowbreak 1,\ldots ,4$ we can formulate the following
Hypothesis concerning general form of them:

\begin{conjecture}
For $n\geq 1$, we have: $A_{0,-\left\lfloor n/2\right\rfloor
+l}^{(n)}\allowbreak =\allowbreak \QATOPD[ ] {n}{l}_{q}\frac{\rho
_{i-1}^{n-l}\left( \rho _{i}^{2}\right) _{n-l}\rho _{i}^{l}\left( \rho
_{i-1}^{2}\right) _{l}}{\left( \rho _{i-1}^{2}\rho _{i}^{2}\right) _{n}},$ $%
l\allowbreak =0,\ldots ,n$. Moreover for $r\allowbreak =\allowbreak 0,\ldots
,\left\lfloor n/2\right\rfloor $ and $l\allowbreak =\allowbreak 0,\ldots
,n-2r$ $A_{r,-\left\lfloor n/2\right\rfloor +r+l}^{(n)}/A_{0,-\left\lfloor
n/2\right\rfloor +l}^{(n)}\allowbreak =\allowbreak \rho _{i-1}^{r}\rho
_{i}^{r}Q_{r,l}\left( q\right) ,$ where $Q_{r,l}\left( q\right) $ is a
polynomial in $q$ with coefficients depending only on $r$ and $l.$
\end{conjecture}

\section{Proofs}

Proof of the Theorem \ref{main} is based on the properties of the following
function $G_{l,k}\left( y,z,t|q\right) \allowbreak =\allowbreak \sum_{m\geq
0}\frac{t^{m}}{\left[ m\right] _{q}!}H_{m+l}\left( y|q\right) H_{m+k}\left(
z|q\right) $. We will need some of its properties. Namely we will prove the
following Proposition which is in fact a generalization and reformulation
(in terms of polynomials $H_{n})$ of an old result of Carlitz. Original
result of Carlitz concerned polynomials $w_{n}(x|q)\allowbreak =\allowbreak
\sum_{i=0}^{n}\QATOPD[ ] {n}{i}_{q}x^{i}$ and expressions of the form $%
\sum_{i=0}^{\infty }\frac{w_{i}\left( x|q\right) w_{i+k}(x|q)t^{i}}{\left(
q\right) _{n}},$ compare \cite{IA}, Exercise 12.3(d) or \cite{Carlitz72}.

\begin{proposition}
\label{iter}i) $\forall k,l\geq 0:G_{k,l}\left( y,z,t|q\right) \allowbreak
=\allowbreak G_{l,k}\left( z,y,t|q\right) $

ii) for $1\leq j\leq k$ : 
\begin{eqnarray}
G_{k,l}(y,z,t|q)\allowbreak &=&\allowbreak \allowbreak
\sum_{i=0}^{j-1}(-1)^{i}\QATOPD[ ] {k}{i}_{q}q^{\binom{i}{2}%
}t^{i}H_{k-i}(y|q)G_{0,i+l}(y,z,t|q)  \label{klna0l} \\
&&+(-1)^{j}q^{\binom{j}{2}}\sum_{i=j}^{k}\QATOPD[ ] {k}{i}_{q}\QATOPD[ ] {i-1%
}{j-1}_{q}t^{i}G_{k-i,i+l}(y,z,t|q).  \notag
\end{eqnarray}

iii) 
\begin{equation}
G_{k,0}\left( y,z,t|q\right) \allowbreak =\allowbreak \sum_{i=0}^{k}(-1)^{i}%
\QATOPD[ ] {k}{i}_{q}q^{\binom{i}{2}}t^{i}H_{k-i}(y|q)G_{0,i}(y,z,t|q).
\label{expansion}
\end{equation}

iv) 
\begin{equation}
G_{k,0}\left( y,z,t|q\right) =\frac{1}{1-q^{k(k-1)}t^{2k}}%
\sum_{i=0}^{k-1}(-1)^{i}q^{\binom{i}{2}}\QATOPD[ ] {k}{i}_{q}t^{i}\tau
_{k,i}\left( y,z,t|q\right) ,  \label{wyraz}
\end{equation}%
where $\tau _{k,i}\left( y,z,t|q\right) \allowbreak =\allowbreak
(H_{k-i}(y|q)G_{0,i}(y,z,t|q)\allowbreak +\allowbreak (-1)^{k}q^{\binom{k}{2}%
}t^{k}H_{k-i}(z|q)G_{i,0}(y,z,t|q)).$
\end{proposition}

\begin{proof}
i) is obvious.

iii) Take $j\allowbreak =\allowbreak k$ and $l=0$ in ii).

ii) To prove (\ref{klna0l}) we will use formula \newline
$H_{n+m}\left( x|q\right) \allowbreak =\allowbreak H_{n}\left( x|q\right)
H_{m}(x|q)-\sum_{j=1}^{\min (n,m)}\QATOPD[ ] {n}{j}_{q}\QATOPD[ ] {m}{j}_{q}%
\left[ j\right] _{q}!H_{n+m-2j}(x|q)$. We have 
\begin{gather*}
G_{k,l}\left( y,z,t\right) \allowbreak =\allowbreak \sum_{m\geq 0}\frac{t^{m}%
}{\left[ m\right] _{q}!}H_{m+k}\left( y|q\right) H_{m+l}\left( z|q\right)
\allowbreak \\
=\sum_{m\geq 0}\frac{t^{m}}{\left[ m\right] _{q}!}(H_{k}\left( y|q\right)
H_{m}(y|q)\allowbreak \\
-\sum_{j=1}^{\min (k,m)}\QATOPD[ ] {k}{j}_{q}\QATOPD[ ] {m}{j}_{q}\left[ j%
\right] _{q}!H_{k+m-2j}(y|q))H_{m+l}\left( z|q\right) \allowbreak \\
=H_{k}\left( y|q\right) G_{0,l}(y,z,t)\allowbreak -\allowbreak
\sum_{j=1}^{k}t^{j}\QATOPD[ ] {k}{j}_{q}\sum_{m=j}^{\infty }\frac{t^{m-j}}{%
\left[ m-j\right] _{q}!}H_{k+(m-j)-j}(y|q)H_{l+j+(m-j)}(z|q) \\
=H_{k}\left( y|q\right) G_{0,l}(y,z,t)\allowbreak -\allowbreak
\sum_{j=1}^{k}t^{j}\QATOPD[ ] {k}{j}_{q}G_{k-j,l+j}(y,z,t).
\end{gather*}%
$\allowbreak $ Hence let us assume that (\ref{klna0l}) is true for $%
j\allowbreak =\allowbreak 1,2,\ldots ,m$. We have after applying just
obtained formula (for $G_{k,l})$ applied for $k->$ $k-m$ and $l->m+l:$ 
\begin{eqnarray*}
G_{k,l}(y,z,t)\allowbreak &=&\allowbreak \allowbreak \sum_{i=0}^{m-1}(-1)^{i}%
\QATOPD[ ] {k}{i}_{q}q^{\binom{i}{2}}t^{i}H_{k-i}(y|q)G_{0,i+l}(y,z,t) \\
&&+\allowbreak (-1)^{m}q^{\binom{m}{2}}\sum_{i=m}^{k}\QATOPD[ ] {k}{i}_{q}%
\QATOPD[ ] {i-1}{m-1}_{q}t^{i}G_{k-i,i+l}(y,z,t)\allowbreak \\
&=&\sum_{i=0}^{m-1}(-1)^{i}\QATOPD[ ] {k}{i}_{q}q^{\binom{i}{2}%
}t^{i}H_{k-i}(y|q)G_{0,i+l}(y,z,t)\allowbreak \\
&&+(-1)^{m}q^{\binom{m}{2}}\QATOPD[ ] {k}{m}%
_{q}t^{m}(H_{k-m}(y|q)G_{0,m+l}(y,z,t) \\
&&-\sum_{i=1}^{k-m}\QATOPD[ ] {k-m}{i}_{q}t^{i}G_{k-m-i,l+m+i}(y,z,t))%
\allowbreak \\
&&+(-1)^{m}q^{\binom{m}{2}}\sum_{i=m+1}^{k}\QATOPD[ ] {k}{i}_{q}\QATOPD[ ] {%
i-1}{m-1}_{q}t^{i}G_{k-i,i+l}(y,z,t)\allowbreak
\end{eqnarray*}%
$\allowbreak $Now since $\QATOPD[ ] {k}{m}_{q}\QATOPD[ ] {k-m}{i-m}_{q}-%
\QATOPD[ ] {k}{i}_{q}\QATOPD[ ] {i-1}{m-1}_{q}\allowbreak =\allowbreak 
\QATOPD[ ] {k}{i}_{q}(\QATOPD[ ] {i}{m}_{q}\allowbreak -\allowbreak \QATOPD[
] {i-1}{m-1}_{q})\allowbreak =\allowbreak q^{m}\QATOPD[ ] {k}{i}_{q}\QATOPD[
] {i-1}{m}_{q}$ and $\binom{m}{2}\allowbreak +\allowbreak m\allowbreak
=\allowbreak \binom{m+1}{2}$ we have 
\begin{eqnarray*}
G_{k,l}(y,z,t)\allowbreak &=&\allowbreak \sum_{i=0}^{m}(-1)^{i}\QATOPD[ ] {k%
}{i}_{q}q^{\binom{i}{2}}t^{i}H_{k-i}(y|q)G_{0,i+l}(y,z,t)\allowbreak
\allowbreak \\
&&-(-1)^{m}q^{\binom{m}{2}}\sum_{i=m+1}^{k}(\QATOPD[ ] {k}{m}_{q}\QATOPD[ ] {%
k-m}{i-m}_{q}-\QATOPD[ ] {k}{i}_{q}\QATOPD[ ] {i-1}{m-1}%
_{q})t^{i}G_{k-i,l+i}(y,z,t) \\
&=&\sum_{i=0}^{m}(-1)^{i}\QATOPD[ ] {k}{i}_{q}q^{\binom{i}{2}%
}t^{i}H_{k-i}(y|q)G_{0,i+l}(y,z,t)\allowbreak \\
&&+(-1)^{m+1}q^{\binom{m+1}{2}}\sum_{i=m+1}^{k}\QATOPD[ ] {k}{i}_{q}\QATOPD[ 
] {i-1}{m}_{q}t^{i}G_{k-i,i+l}(y,z,t)
\end{eqnarray*}%
iv) For $k\allowbreak =\allowbreak 0$ this is obviously true. Now let us
iterate (\ref{expansion}) once, applied however, for $G_{0,k}$. We will get
then \newline
$G_{k,0}\left( y,z,t|q\right) \allowbreak =\allowbreak
\sum_{i=0}^{k-1}(-1)^{i}\QATOPD[ ] {k}{i}_{q}q^{\binom{i}{2}%
}t^{i}H_{k-i}(y|q)G_{0,i}(y,z,t|q)\allowbreak +\allowbreak \left( -1\right)
^{k}q^{\binom{k}{2}}t^{k}G_{0,k}\left( y,z,t|q\right) \allowbreak =$\newline
$\allowbreak \sum_{i=0}^{k-1}(-1)^{i}\QATOPD[ ] {k}{i}_{q}q^{\binom{i}{2}%
}t^{i}H_{k-i}(y|q)G_{0,i}(y,z,t|q)\allowbreak +\allowbreak $\newline
$\left( -1\right) ^{k}q^{\binom{k}{2}}t^{k}(\sum_{i=0}^{k-1}(-1)^{i}\QATOPD[
] {k}{i}_{q}q^{\binom{i}{2}}t^{i}H_{k-i}(z|q)G_{i,0}(y,z,t|q))\allowbreak
+\allowbreak q^{k(k-1)}t^{2k}G_{k.0}\left( y,z,t|q\right) $. Thus we see
that since for all $i\leq k-1$ $G_{i,0}$ and $G_{0,i}$ and are of the
claimed form then from (\ref{wyraz}) it follows that $G_{k,0}$ has the
claimed form.
\end{proof}

Now we are ready to present the proof if Theorem \ref{main}.

\begin{proof}[Proof of the Theorem \protect\ref{main}]
To prove i) we will use formula viii) of Lemma \ref{wlasnosci}, that is
Poisson-Mehler expansion formula. Following (\ref{_x|yz}) we see that 
\begin{eqnarray*}
\phi \left( x|y,z,\rho _{1},\rho _{2},q\right) \allowbreak &=&\allowbreak
f_{N}\left( x|q\right) \sum_{i=0}^{\infty }\frac{\rho _{1}^{i}}{\left[ i%
\right] _{q}!}H_{i}\left( x|q\right) H_{i}\left( y|q\right) \allowbreak
\times \allowbreak \sum_{i=0}^{\infty }\frac{\rho _{2}^{i}}{\left[ i\right]
_{q}!}H_{i}\left( x|q\right) H_{i}\left( z|q\right) \allowbreak \\
&&/\allowbreak \sum_{i=0}^{\infty }\frac{\rho _{1}^{i}\rho _{2}^{i}}{\left[ i%
\right] _{q}!}H_{i}\left( y|q\right) H_{i}\left( z|q\right) .
\end{eqnarray*}%
First, let us concentrate on the quantity: 
\begin{equation*}
R\left( x,y,z,\rho _{1},\rho _{2}|q\right) \allowbreak =\allowbreak
\sum_{n\geq 0}\frac{\rho _{1}^{n}}{\left[ n\right] _{q}!}H_{n}\left(
x|q\right) H_{n}\left( y|q\right) \allowbreak \times \allowbreak \sum_{m\geq
0}\frac{\rho _{2}^{n}}{\left[ n\right] _{q}!}H_{n}\left( x|q\right)
H_{n}\left( z|q\right) .
\end{equation*}%
We will apply identity (\ref{identity}), distinguish two cases $n+m$ is even
and $n+m$ is odd$,$ denote $n+m-2j\allowbreak =\allowbreak $ $2k$ or $%
n+m-2j\allowbreak =\allowbreak 2k+1$ depending om the case and sum over the
set of $\{(n,m)\allowbreak :\allowbreak n+m-2k\allowbreak \leq \allowbreak
2\min (n,m),\allowbreak m,n\geq 0\}\allowbreak \cup $\allowbreak $%
\{(n,m)\allowbreak :\allowbreak n+m-2k-1\allowbreak \leq \allowbreak 2\min
(n,m),\allowbreak m,n\geq 0\}$. \allowbreak We have 
\begin{gather*}
R\left( x,y,z,\rho _{1},\rho _{2}|q\right) \allowbreak =\allowbreak
\sum_{n,m\geq 0}\frac{\rho _{1}^{n}\rho _{2}^{m}}{\left[ n\right] _{q}!\left[
m\right] _{q}!}H_{n}\left( y|q\right) H_{m}\left( z|q\right) \allowbreak
\sum_{j=0}^{\min \left( n,m\right) }\QATOPD[ ] {n}{j}_{q}\QATOPD[ ] {m}{j}%
_{q}\left[ j\right] _{q}!H_{n+m-2j}\left( x|q\right) \\
=\sum_{k=0}^{\infty }\frac{H_{2k}(x|q)}{\left[ 2k\right] _{q}!}\allowbreak
\sum_{j=0}^{\infty }\allowbreak \sum_{m=j}^{2k+j}\frac{\left[ 2k\right]
_{q}!\rho _{1}^{m}\rho _{2}^{2k+2j-m}}{\left[ m-j\right] _{q}!\left[ 2k-(m-j)%
\right] _{q}!\left[ j\right] _{q}}H_{m}(y|q)H_{2k+2j-m}(z|q)\allowbreak \\
+\sum_{k=0}^{\infty }\frac{H_{2k+1}(x|q)}{\left[ 2k+1\right] _{q}!}%
\allowbreak \sum_{j=0}^{\infty }\allowbreak \sum_{m=j}^{2k+1+j}\frac{\left[
2k+1\right] _{q}!\rho _{1}^{i+j}\rho _{2}^{2k+1-i+j}}{\left[ j\right] _{q}%
\left[ m-j\right] _{q}!\left[ 2k+1-(m-j)\right] _{q}!}H_{m}\left( y|q\right)
H_{2k+1+2j-m}\left( z|q\right) \\
=\sum_{k=0}^{\infty }\frac{H_{2k}(x|q)}{\left[ 2k\right] _{q}!}\allowbreak
\sum_{j=0}^{\infty }\allowbreak \sum_{i=0}^{2k}\frac{\left[ 2k\right]
_{q}!\rho _{1}^{i+j}\rho _{2}^{2k-i+j}}{\left[ j\right] _{q}!\left[ i\right]
_{q}![2k-i]_{q}!}H_{i+j}\left( y|q\right) H_{2k-i+j}\left( z|q\right)
\allowbreak \\
+\sum_{k=0}^{\infty }\frac{H_{2k+1}(x|q)}{\left[ 2k+1\right] _{q}!}%
\allowbreak \sum_{j=0}^{\infty }\allowbreak \sum_{i=0}^{2k+1}\frac{\left[
2k+1\right] _{q}!\rho _{1}^{i+j}\rho _{2}^{2k+1-i+j}}{\left[ j\right] _{q}!%
\left[ i\right] _{q}!\left[ 2k+1-i\right] _{q}!}H_{i+j}\left( y|q\right)
H_{2k+1+j-i}\left( z|q\right)
\end{gather*}%
$\allowbreak $

We get then%
\begin{gather*}
R\left( x,y,z,\rho _{1},\rho _{2}|q\right) =\sum_{k=0}^{\infty }\frac{%
H_{2k}(x|q)}{\left[ 2k\right] _{q}!}\allowbreak \sum_{i=0}^{2k}\QATOPD[ ] {2k%
}{i}_{q}\rho _{1}^{i}\rho _{2}^{2k-i}\sum_{j=0}^{\infty }\frac{(\rho
_{1}\rho _{2})^{j}}{\left[ j\right] _{q}!}H_{i+j}(y|q)H_{2k-i+j}(z|q)%
\allowbreak \\
+\sum_{k=0}^{\infty }\frac{H_{2k+1}(x|q)}{\left[ 2k+1\right] _{q}!}%
\sum_{i=0}^{2k+1}\QATOPD[ ] {2k+1}{i}_{q}\rho _{1}^{i}\rho
_{2}^{2k+1-i}\sum_{j=0}^{\infty }\frac{(\rho _{1}\rho _{2})^{j}}{\left[ j%
\right] _{q}!}H_{i+j}(y|q)H_{2k+1-i+j}(z|q) \\
=\sum_{n=0}^{\infty }\frac{H_{n}\left( x|q\right) }{\left[ n\right] _{q}!}%
\allowbreak \sum_{i=0}^{n}\QATOPD[ ] {n}{i}_{q}\rho _{1}^{i}\rho
_{2}^{n-i}\allowbreak \sum_{j=0}^{\infty }\frac{(\rho _{1}\rho _{2})^{j}}{%
\left[ j\right] _{q}!}H_{i+j}(y|q)H_{n-i+j}(z|q)
\end{gather*}%
$\allowbreak $\newline
Using introduced in Proposition \ref{iter} function $G_{l,k}\left(
y,z,t|q\right) \allowbreak =\allowbreak \sum_{m\geq 0}\frac{t^{m}}{\left[ m%
\right] _{q}!}H_{m+l}\left( y|q\right) H_{m+k}\left( z|q\right) $ we can
express both \newline
$\sum_{i=0}^{\infty }\frac{\rho _{1}^{i}\rho _{2}^{i}}{\left[ i\right] _{q}!}%
H_{i}\left( y|q\right) H_{i}\left( z|q\right) \allowbreak =\allowbreak
G_{0,0}(y,z,\rho _{1}\rho _{2}|q)$ and \newline
$R\left( x,y,z,\rho _{1},\rho _{2}|q\right) \allowbreak =\allowbreak
\sum_{n=0}^{\infty }\frac{H_{n}\left( x|q\right) }{\left[ n\right] _{q}!}%
\sum_{i=0}^{n}\QATOPD[ ] {n}{i}_{q}\rho _{1}^{i}\rho
_{2}^{n-i}G_{i,n-i}\left( y,z,\rho _{1}\rho _{2}|q\right) $. Our Theorem
will be proved if we will be able to show that $\forall l,k\geq 0:$ $%
G_{l,k}\left( y,z,t|q\right) \allowbreak =\allowbreak G_{0,0}\left(
y,z,t|q\right) \Theta _{l,k}\left( y,z,t|q\right) $ where $\Theta _{l,k}$ is
a polynomial of order a $l$ in $y$ $k$ in $z$. This fact follows by
induction from formula (\ref{wyraz}) of assertion iv) of the Proposition \ref%
{iter} since it expresses $G_{k,0}$ in terms of $k$ functions $G_{i,0}$ and $%
G_{0,i}$ for $i=0,\ldots ,k-1$ and the fact that all $G_{l,k}$ can be
expressed by $G_{i,0}$ and $G_{0,i}$ ; $i\leq k+l.$
\end{proof}

\begin{proof}[Proof of Corollary \protect\ref{postac}]
By Theorem \ref{main} we know that regression \newline
$\mathbb{E(}H_{n}\left( X_{i}|q\right) |X_{j_{1}},\ldots
,X_{j_{k}},X_{j_{m}},\ldots ,X_{j_{h}})$ is a polynomial in $X_{j_{k}}$ and $%
X_{j_{m}}$ of order at most $n$. To analyze the structure of this polynomial
let us present it in the form $\sum_{s=0}^{n}a_{s,n}H_{s}\left(
X_{j_{m}}|q\right) $ where coefficients $a_{s,n}$ are some polynomials of $%
X_{j_{k}}$. Now let us take conditional expectation with respect to $%
X_{j_{1}},\ldots ,X_{j_{k}}$ of both sides. On one hand we get\newline
\begin{equation*}
\mathbb{E}\left( H_{n}\left( X_{i}|q\right) |X_{j_{1}},\ldots
,X_{j_{k}}\right) \allowbreak =\allowbreak \left(
\prod_{m=j_{k}}^{i-1}\sigma _{m}\right) ^{n}H_{n}\left( X_{j_{k}}\right)
\end{equation*}
on the other we get\newline
\begin{equation*}
\sum_{s=0}^{n}a_{s,n}\left( \prod_{m=j_{k}}^{j_{m-1}}\sigma _{m}\right)
^{s}H_{s}\left( X_{jk}|q\right) .
\end{equation*}
Since $a_{s,n}$ are polynomials in $X_{j_{k}}$ of order at most $n,$ we can
present them in the form 
\begin{equation*}
a_{s,n}\allowbreak =\allowbreak \sum_{t=0}^{n}\beta _{t,s}H_{t}\left(
x_{j_{k}}|q\right) .
\end{equation*}
Thus we have equality:\newline
\begin{equation*}
\left( \prod_{m=j_{k}}^{i-1}\sigma _{m}\right) ^{n}H_{n}\left(
x_{j_{k}}\right) \allowbreak =\allowbreak \sum_{s=0}^{n}\left(
\prod_{m=j_{k}}^{j_{m-1}}\sigma _{m}\right) ^{s}\allowbreak
\sum_{t=0}^{n}\beta _{t,s}H_{t}\left( x_{j_{k}}\right) H_{s}\left(
x_{j_{k}}\right) .
\end{equation*}
Now we use the identity (\ref{identity}) and get 
\begin{gather*}
\left( \prod_{m=j_{k}}^{i-1}\sigma _{m}\right) ^{n}H_{n}\left(
x_{j_{k}}|q\right) \allowbreak =\allowbreak
\sum_{s=0}^{n}\sum_{t=0}^{n}\beta _{t,s}\left(
\prod_{m=j_{k}}^{j_{m-1}}\sigma _{m}\right) ^{s}\allowbreak \\
\times \sum_{m=0}^{\min \left( t,s\right) }\QATOPD[ ] {t}{m}_{q}\QATOPD[ ] {s%
}{m}_{q}\left[ m\right] _{q}!H_{t+s-2m}\left( x_{j_{k}}|q\right) .
\end{gather*}
Hence we deduce that $\beta _{t,s}\allowbreak =\allowbreak 0$ for $t+s>n$ , $%
t+s\allowbreak \allowbreak =\allowbreak n-1,n-3,\ldots ,$. To count the
number of coefficients $A_{j,k}^{\left( n\right) }$ observe that we have $%
n+1 $ coefficients $A_{0,k}^{\left( n\right) }$ since $k$ ranges from $%
-\left\lfloor \frac{n}{2}\right\rfloor $ to $-\left\lfloor \frac{n}{2}%
\right\rfloor +n$ , $n-1$ coefficients $A_{1,k}^{\left( n\right) }$ where $k$
ranges from $-\left\lfloor n/2\right\rfloor +1$ to $-\left\lfloor
n/2\right\rfloor +n-1$ and so on.
\end{proof}

\begin{proof}[Proof of Corollary \protect\ref{_3i4}]
The proof is based on the idea of writing down system of $\left\lfloor \frac{%
n+2}{2}\right\rfloor \left\lfloor \frac{n+3}{2}\right\rfloor $ ( $%
n\allowbreak =\allowbreak 1,\ldots ,4$ ) linear equations satisfied by
coefficients $A_{m,k}^{(n)}$. These equations are obtained according to the
similar pattern. Namely we multiply both sides of identity (\ref{wzor}) by $%
H_{m}\left( X_{i-1}\right) $ and $H_{k}\left( X_{i}\right) $ and calculate
conditional expectation of both sides with respect to $\mathcal{F}_{<i}$ or
with respect to $\mathcal{F}_{>i}$ remembering that $\mathbb{E}\left(
H_{m}\left( X_{i+1}\right) |\mathcal{F}_{<i}\right) \allowbreak =\allowbreak
\rho _{i-1}^{m}\rho _{i}^{m}H_{m}\left( X_{i-1}\right) $ and $\mathbb{E}%
\left( H_{m}\left( X_{i}\right) |\mathcal{F}_{<i}\right) \allowbreak
=\allowbreak \rho _{i-1}^{m}H_{m}\left( X_{i-1}\right) $ and similar
formulae for $\mathcal{F}_{>i}$. We expand both sides with respect to $%
H_{s}, $ $\allowbreak s=\allowbreak n+m+k-2t,$ $t\allowbreak =\allowbreak
0,\ldots ,\left\lfloor (n+m+k)/2\right\rfloor $. On the way we utilize (\ref%
{identity}) and compare coefficients standing by $H_{s}$ on both sides. Thus
each obtained equation involving coefficients $A_{i,j}^{\left( n\right) }$
can be indexed by $s,m,j$ and $r$ if we calculate conditional expectation
with respect to $\mathcal{F}_{<i}$ of $l$ if we conditional expectation is
calculated with respect to $\mathcal{F}_{>i}$. Of course if $s\allowbreak
=\allowbreak 0$ then $r$ and $l$ lead to the same result. Formulae for $%
A_{i,j}^{\left( n\right) },$ for $n\allowbreak =\allowbreak 1$ are obtained
by taking $s\allowbreak =\allowbreak 1,$ $m\allowbreak =\allowbreak 0,$ $%
j\allowbreak =\allowbreak 0$ and applying $r$ and $l$. For $n\allowbreak
=\allowbreak 2$ first we consider $m\allowbreak =\allowbreak 0,$ $%
j\allowbreak =\allowbreak 0$ and $s\allowbreak =\allowbreak 2$ and applying $%
r$ and $l$ and then $m\allowbreak =\allowbreak 0,$ $j\allowbreak
=\allowbreak 0$ and $s\allowbreak =\allowbreak 0$. In this way we get $3$
equations. The forth one is obtained by taking $m\allowbreak =\allowbreak 0,$
$j\allowbreak =\allowbreak 1,$ $s=1$ and $r$. Denote $\mathbf{X\allowbreak
=\allowbreak (}A_{0,-1}^{(2},A_{0,0}^{(2)},A_{0,1}^{(2)},A_{1,0}^{(2)})^{T},$
then $\mathbf{X}$ satisfies system of linear equation with matrix $\left( 
\begin{array}{cccc}
1 & \text{$\rho $}\text{$_{i-1}\rho _{i}$} & \text{$\rho $}_{i-1}^{2}\text{$%
\rho $}_{i}^{2} & 0 \\ 
\text{$\rho $}_{i-1}^{2}\text{$\rho $}_{i}^{2} & \text{$\rho $}\text{$%
_{i-1}\rho _{i}$} & 1 & 0 \\ 
0 & \text{$\rho $}\text{$_{i-1}\rho _{i}$} & 0 & 1 \\ 
\left[ 2\right] _{q}\text{$\rho $}\text{$_{i-1}\rho _{i}$} & 1+\left[ 2%
\right] _{q}\text{$\rho $}_{i-1}^{2}\text{$\rho $}_{i}^{2} & \left[ 2\right]
_{q}\text{$\rho $}\text{$_{i-1}\rho _{i}$} & \text{$\rho $}\text{$_{i-1}\rho
_{i}$}%
\end{array}%
\right) $ \newline
and with right side vector equal to :$\left( 
\begin{array}{c}
\rho _{i-1}^{2} \\ 
\text{$\rho $}_{i}^{2} \\ 
0 \\ 
\left[ 2\right] _{q}\text{$\rho $}\text{$_{i-1}\rho _{i}$}%
\end{array}%
\right) $. Besides formulae for coefficients $A_{m,k}^{(n)},$ for $%
n\allowbreak =\allowbreak 1,2$ can be obtained from formulae scattered in
the literature like e.g. \cite{bryc1}, \cite{matszab2} or \cite{Szab3}. To
get equations satisfied by coefficients $A_{i,j}^{\left( n\right) },$ for $%
n\allowbreak =\allowbreak 3,4$ First $n+1$ equations are obtained by taking $%
m\allowbreak =\allowbreak 0$ , $k\allowbreak =\allowbreak 0$ and $s=$%
\allowbreak $3,1$ if $n\allowbreak =\allowbreak 3$ and $s=$\allowbreak $%
4,2,0 $ if $n\allowbreak =\allowbreak 4$ and then applying operations $r$
and $l$. Then, in order to get remaining $2$ (in case of $n=3)$ or $4$ (in
case of $n\allowbreak =\allowbreak 4)$ equations one has to be more careful
since it often turns out that many equations obtained for some $m$ and $k$
are linearly dependent on the previously obtained equations. In the case of $%
n\allowbreak =\allowbreak 3$ to get remaining two linearly independent
equations we took $m\allowbreak =\allowbreak 2,$ $k\allowbreak =\allowbreak
0,$ $s\allowbreak =\allowbreak 3$ and applied operations $r$ and $l$. In
this way we obtained system of $6$ linear equations with matrix $\left( 
\begin{array}{cccccc}
1 & \text{$\rho _{i-1}$}\text{$\rho _{i}$} & \text{$\rho $}_{i-1}^{2}\text{$%
\rho $}_{i}^{2} & \text{$\rho $}_{i-1}^{3}\text{$\rho $}_{i}^{3} & 0 & 0 \\ 
\text{$\rho $}_{i-1}^{3}\text{$\rho $}_{i}^{3} & \text{$\rho $}_{i-1}^{2}%
\text{$\rho $}_{i}^{2} & \text{$\rho _{i-1}$}\text{$\rho _{i}$} & 1 & 0 & 0
\\ 
0 & (1+q)\text{$\rho _{i-1}$}\text{$\rho _{i}$} & (1+q)\text{$\rho $}%
_{i-1}^{2}\text{$\rho $}_{i}^{2} & 0 & 1 & \text{$\rho _{i-1}$}\text{$\rho
_{i}$} \\ 
0 & (1+q)\text{$\rho $}_{i-1}^{2}\text{$\rho $}_{i}^{2} & (1+q)\text{$\rho
_{i-1}$}\text{$\rho _{i}$} & 0 & \text{$\rho _{i-1}$}\text{$\rho _{i}$} & 1
\\ 
\left[ 3\right] _{q}\text{$\rho _{i-1}$}\text{$\rho _{i}$} & 1+\left[ 2%
\right] _{q}^{2}\text{$\rho $}_{i-1}^{2}\text{$\rho $}_{i}^{2} & \left[ 2%
\right] _{q}\text{$\rho _{i-1}$}\text{$\rho _{i}$}+\left[ 3\right] _{q}\text{%
$\rho $}_{i-1}^{3}\text{$\rho $}_{i}^{3} & \left[ 3\right] _{q}\text{$\rho $}%
_{i-1}^{2}\text{$\rho $}_{i}^{2} & \text{$\rho _{i-1}$}\text{$\rho _{i}$} & 
\text{$\rho $}_{i-1}^{2}\text{$\rho $}_{i}^{2} \\ 
\left[ 3\right] _{q}\text{$\rho $}_{i-1}^{2}\text{$\rho $}_{i}^{2} & \left[ 2%
\right] _{q}\text{$\rho _{i-1}$}\text{$\rho _{i}$}+\left[ 3\right] _{q}\text{%
$\rho $}_{i-1}^{3}\text{$\rho $}_{i}^{3} & 1+\left[ 2\right] _{q}^{2}\text{$%
\rho $}_{i-1}^{2}\text{$\rho $}_{i}^{2} & \left[ 3\right] _{q}\text{$\rho
_{i-1}$}\text{$\rho _{i}$} & \text{$\rho $}_{i-1}^{2}\text{$\rho $}_{i}^{2}
& \text{$\rho _{i-1}$}\text{$\rho _{i}$}%
\end{array}%
\right) $\newline
right hand side vector $\left( 
\begin{array}{c}
\rho _{i-1}^{3} \\ 
\text{$\rho $}_{i}^{3} \\ 
0 \\ 
0 \\ 
\left[ 3\right] _{q}\text{$\rho _{i-1}^{2}$}\text{$\rho _{i}$} \\ 
\left[ 3\right] _{q}\text{$\rho $}_{i-1}\text{$\rho $}_{i}^{2}%
\end{array}%
\right) $ if the vector of unknowns is the following $(A_{0,-1}^{(3)},\ldots
A_{0,2}^{(2)},A_{1,0}^{(3)},A_{1,1}^{(3)})^{T}$. For $n\allowbreak
=\allowbreak 4$ remaining $4$ equations we obtained by taking: ($%
m\allowbreak =\allowbreak 1$ , $k\allowbreak =\allowbreak 4,$ $s$\allowbreak 
$=\allowbreak 3,$ $r),$ ($m\allowbreak =\allowbreak 4,$ $k\allowbreak
=\allowbreak 1,$ $s$\allowbreak $=\allowbreak 1,$ $r),$ ($m\allowbreak
=\allowbreak 4,$ $k\allowbreak =\allowbreak 2,$ $s\allowbreak =\allowbreak
4, $ $r)$ and $(m\allowbreak =\allowbreak 2,$ $k\allowbreak =\allowbreak 4,$ 
$s=4,,$ $r)$. Recall that in this case we have $9$ equations. Matrix of this
system has $81$ entries. That is why we will skip writing down the whole
system of equations. To get the scent of how complicated these equations are
we will present one equation. \text{For }$n\allowbreak =\allowbreak 4$ one
of the equations (referring to the case $m\allowbreak =\allowbreak 4,$ $k=1,$
$s=1,$ $r)$ is $\left[ 2\right] _{q}^{2}\left( 1+q^{2}\right) \left[ 3\right]
_{q}(\left[ 4\right] _{q}+\left[ 5\right] _{q}))$\text{$\rho _{i-1}\rho
_{i}A_{0,-2}^{\left( 4\right) }\allowbreak +\allowbreak $}$\left[ 2\right]
_{q}^{2}\left( 1+q^{2}\right) \left[ 3\right] _{q}\allowbreak \times $%
\allowbreak $(1\allowbreak +\allowbreak \rho _{i-1}^{2}\rho
_{i}^{2}\allowbreak +\allowbreak 3q\rho _{i-1}^{2}\rho _{i}^{2}\allowbreak
+\allowbreak 3q^{2}\rho _{i-1}^{2}\rho _{i}^{2}\allowbreak +\allowbreak
q^{3}\rho _{i-1}^{2}\rho _{i}^{2}\allowbreak +\allowbreak q^{4}\rho
_{i-1}^{2}\rho _{i}^{2})A_{0,-1}^{\left( 4\right) }\allowbreak +\allowbreak %
\left[ 2\right] _{q}^{2}\left( 1+q^{2}\right) \left[ 3\right] _{q}\rho
_{i-1}\rho _{i}(\left[ 2\right] _{q}\allowbreak +\allowbreak \rho
_{i-1}^{2}\rho _{i}^{2}\allowbreak +\allowbreak 3q\rho _{i-1}^{2}\rho
_{i}^{2}\allowbreak +\allowbreak 3q^{2}\rho _{i-1}^{2}\rho
_{i}^{2}\allowbreak +\allowbreak q^{3}\rho _{i-1}^{2}\rho
_{i}^{2}\allowbreak +\allowbreak q^{4}\rho _{i-1}^{2}\rho
_{i}^{2})A_{0,0}^{\left( 4\right) }\allowbreak +\allowbreak \left[ 2\right]
_{q}^{2}\left( 1+q^{2}\right) \left[ 3\right] _{q}\rho 1^{2}\rho 2^{2}(\left[
3\right] _{q}+(\left[ 4\right] _{q}+\left[ 5\right] _{q})\rho _{i-1}^{2}\rho
_{i}^{2})A_{0,1}^{\left( 4\right) }\allowbreak +\allowbreak \left[ 2\right]
_{q}^{2}\left( 1+q^{2}\right) \left[ 3\right] _{q}\rho 1^{3}\rho
2^{3}(1+q+q^{2}+q^{3}+\rho _{i-1}^{2}\rho _{i}^{2}+q\rho _{i-1}^{2}\rho
_{i}^{2}+q^{2}\rho _{i-1}^{2}\rho _{i}^{2}+q^{3}\rho _{i-1}^{2}\rho
_{i}^{2}+q^{4}\rho _{i-1}^{2}\rho _{i}^{2})A_{0,2}^{(4)}\allowbreak
+\allowbreak $

$\left[ 2\right] _{q}^{2}\left( 1+q^{2}\right) \left[ 3\right] _{q}\rho
_{i-1}\rho _{i}A_{1,-1}^{\left( 4\right) }\allowbreak +\allowbreak \left[ 2%
\right] _{q}^{2}\left( 1+q^{2}\right) \left[ 3\right] _{q}\rho
_{i-1}^{2}\rho _{i}^{2}A_{1,0}^{\left( 4\right) }\allowbreak +\allowbreak %
\left[ 2\right] _{q}^{2}\left( 1+q^{2}\right) \left[ 3\right] _{q}\rho
_{i-1}^{3}\rho _{i}^{3}A_{1,1}^{\left( 4\right) }\allowbreak =\allowbreak $ $%
\left[ 2\right] _{q}^{2}\left( 1+q^{2}\right) \left[ 3\right] _{q}\rho
_{i-1}^{3}(\left[ 4\right] _{q}+\left[ 5\right] _{q}\rho _{i-1}^{2})\rho
_{i}.$
\end{proof}

\begin{acknowledgement}
The author would like to thank the referee for his many precise, valuable
remarks that helped to improve the paper.
\end{acknowledgement}

\end{document}